\documentclass[leqno,11pt]{amsart}
\pdfoutput=1		

\usepackage[bookmarksnumbered]{hyperref}	
\usepackage[ngerman,USenglish]{babel}			
\usepackage{color,a4}											
\usepackage[ascii]{inputenc}								
\usepackage{aliascnt}										
\usepackage{microtype}									
\usepackage[all]{hypcap}
\usepackage{tikz}											
\usepackage{verbatim}

\newcommand{\newjointcountertheorem}[3]{
	\newaliascnt{#1}{#2}
	\newtheorem{#1}[#1]{#3}
	\aliascntresetthe{#1}	
}

\newtheorem{thm}{Theorem}[section]
\newjointcountertheorem{satz}{thm}{Satz}
\newjointcountertheorem{lem}{thm}{Lemma}
\newjointcountertheorem{cor}{thm}{Corollary}
\newjointcountertheorem{prp}{thm}{Proposition}
\newjointcountertheorem{cnj}{thm}{Conjecture}
\newjointcountertheorem{que}{thm}{Question}
\newjointcountertheorem{fct}{thm}{Fact}
\theoremstyle{definition}
\newjointcountertheorem{dfn}{thm}{Definition}
\newjointcountertheorem{ntn}{thm}{Notation}
\newjointcountertheorem{rem}{thm}{Remark}
\newjointcountertheorem{nte}{thm}{Note}
\newjointcountertheorem{exl}{thm}{Example}

\def\Snospace~{\S{}}

\renewcommand{\mathbb}{\mathbf}

\newcommand{\var}{{\mathrm{Var}}}

\def\lra{\longrightarrow }

\def\var{\operatorname{Var}}
\def\eins{^{(1)}}
\def\zwei{^{(2)}}

\def\xnk{\frac{k!\,(n-k)!}{n!}\,[x^{k}] \;}

\def\tno{\tfrac{k!\,(n-k)!}{n!}\,[ t^k]\;}
\def\tn2{\tfrac{k!}{n^k}\,[ t^k]\;}

\def\ci{\cite}

\def\lra{\longrightarrow}
\def\PP{\mathbb{P}}
\def\EE{\mathbb{E}}
\numberwithin{equation}{section}
\allowdisplaybreaks[4]							

\begin{document}

	\title[{On (multi-) collision times}]{On (Multi)-Collision Times}

	\author{Ernst Schulte-Geers}
	\address{Bundesamt f\"ur Sicherheit in der Informationstechnik, Godesberger Allee 185--189, 53175 Bonn, Germany}
	\email{ernst.schulte-geers@bsi.bund.de}
	
	\date{\today}

\begin{abstract}
We study the (random) waiting time for the appearance of the first (multi-)collision in a drawing process in detail. 
The results have direct implications for the assessment of generic (multi-)collision search in cryptographic hash functions.
\end{abstract}
	\maketitle

	\bibliographystyle{amsplain}

\section{Introduction}
A $(n,m)-$ function is a function $h:\mathcal{D} \lra \mathcal{R}$ where $|{\mathcal D}|=n$ and $|{\mathcal R}|=m$ are finite sets.\\
An $r$-fold multi-collision (short: $r$-collision) of $h$ is an $r-$element set $\{d_1,\ldots,d_r\}$ of (mutually distinct) domain points
s.th. $h(d_1)=h(d_2)=\ldots =h(d_r)$, 2-collisions are called collisions.\\
In particular, for cryptographic hash functions $h$ the difficulty of statistical (multi-) collision search for $h$ is of interest : how difficult is it to find a point 
$y\in \mathcal{R}$ and two  (resp. $r$) different $h$-preimages of  $y$?.\\
In a generic statistical (multi-) collision attack on a hash function $h$ an attacker produces randomly  hash values (i.e. he produces randomly preimages of $h$ and maps them through $h$) until he has found the first (multi-) collision.\\
How many hash values must be produced to find a (multi-)collision? \\
This number is a random variable - the waiting time $K_r$ (resp. $R_r$) for the first (multi-) collision (resp. repetition) - and the distribution of this ``collision time" describes statistically the effort
needed for the (multi-)collision search.\\
The main questions of interest are
\begin{enumerate}
\item  what is  the average effort for the attack (the expectation of the collision time)?
\item what is the typical effort for the attack (the distribution of the collision time)?
\end{enumerate}
Cryptologic ``folklore" states that this question reduces to the classical birthday phenomenon in the codomain, and that  
the ``birthday effort"- i.e. the trial of magnitude $\sqrt{m}$ (resp. $m^{{(r-1)}/r}$ preimages) - is needed to find a (multi-)collision.\\[0.1cm]
The folklore result relies on the underlying assumption that the mapping $h$ behaves as a (uniform) random mapping (the individual domain points ``pick"
their image uniformly at random, indepently of the other domain points).\\[0.1cm]
It is clear - but frequently not explicitly stated - that this is an approximation.\\[0.1cm]
This ``random mapping" approximation
was (for the collision case) questioned by Bellare and Kohno (\cite{BK}), who pointed out that the behaviour of a fixed hash function could possibly deviate strongly from random.
Their standpoint: for a given hash function one should try to quantify resistance to generic collision attacks using a balance measure 
(rather than to assume random mapping behaviour a priori). Similar balance measures for multi-collisions were proposed and investigated by Ramanna and Sarkar in \cite{RS}. These papers 
also provide ``rough" statistical answers to 1. and 2. However, despite of the interest in generic attacks there is apparently a lack of rigorous results so far.\\[0.1cm]
The object of this paper is to close this gap. We prove precise statistical assertions to questions (1) and (2) and give a survey of exact results on the subject.\\[0.1cm]
In the sequel we consider both commonly used models for the mapping $h$: 
\\ (a) the mapping $h$ is fixed  (b) the mapping $h$ is chosen at random.\\
Further we treat both, sampling with replacement and sampling without replacement.

\section{Previous and related work}
Statistical collision search may be thought of as a generalised ``birthday problem". (The exact relationship to
the classical birthday problem is explained in \autoref{ssec:stochmod} below.)
\subsection{ The classical birthday problem }
Let $r\geq 2$. The classical $r$-fold birthday problem deals with the following question:\\
{\em $k$ balls are distributed at random into $m$ cells. How big is the chance that
no cell contains more than $r-1$ balls?}\\
The classical answer, due to von Mises \cite{vonMises}, is as follows:\\[0.1cm]
For $j\in\{0,\ldots,k\}$ let  $U_j^{k,m}$ be the random variable ``number of cells containing exactly $j$ balls".

Then the following theorem holds:
\begin{thm}\label{thmvM}
{\bf (von Mises)}\\ 
Let $j\in\{0,\ldots,k\}$ be fixed and $\alpha=\frac{k}{m}$. Then\\[0.1cm]
(a)  $\EE(U_j^{k,m})=m{k \choose j}(\frac{1}{m})^j(1-\frac{1}{m})^{k-j}\;,$\\[0.1cm]
(b) If $k,m\lra \infty$ s.th. $m\frac{\alpha^j}{j!} e^{-\alpha}$ tends to a finite positive limit $a_j$, then  asymptotically $U_j^{k,m}$ is Poisson distributed with parameter
$a_j$.\end{thm}
The random variable $S_r^{k,m}:=\sum_{i=r}^n {i \choose r}\,U_i^{k,m}$ gives the number of $r-$collisions. Using $1.1.(a)$ we find that
$\EE(S_r^{k,m})=\frac{1}{m^{r-1}}{k \choose r}$. 

 If we regard $k$ as a time variable 
(i.e.: consider the ``cell occupancy process", where the balls are distributed one after another into
the cells) the first $r$-collisions will thus appear at times of order $t_{r,m}:=(r!\,m^{r-1})^{1/r}$./
More precisely:  define the appearance time $T_{r,m}$ of the the first $r$-collision by 
$$\{T_{r,m}>k\}=\{S_r^{k,m}=0\}$$ Then theorem \ref{thmvM} has the following corollary:

\begin{cor}\label{corvM} Let $r\geq 2$ be fixed and $x>0$. Then\\[0.1cm]
$\PP(T_{r,m}/t_{r,m}> x)\lra e^{-x^r}\;\;\;(m\lra \infty) $
\end{cor}
\begin{proof}

 Let $x>0$  and in the sequel $k=k(x,m):=\lceil x\,t_{r,m}\rceil$ and $m\lra \infty$. Then $0\leq S_r^{k,m} -U_r^{k,m}\lra 0$ 
(since $\EE(S_r^{k,m}-U_r^{k,m})=\frac{1}{m^{r-1}}{k \choose r}\left(1-(1-\frac{1}{m})^{k-r}\right)\lra 0$). Thus $S_r^{k,m}$ and $U_r^{k,m}$ have the same limit distribution. 
The conditions of \ref{thmvM},(b) are fulfilled with $a_r=x^r$. Thus
$\PP(T_{r,m}>k(x,m))=\PP(S_r^{k(x,m),m}=0)\lra e^{-x^r}$ 
\end{proof}
The limiting distribution is a Weibull distribution with form parameter $r$. In particular, the case $r=2$ (the Weibull distributions with $r=2$ are 
also called Rayleigh distributions) of this corollary is well known.
\subsection{ A (very incomplete) guide to the literature} $ $\\
These classical results have been generalised and refined in many ways, and
there is an abundance of literature on the classical birthday problem and its generalisations.\\
It is impossible to give a complete survey, we just point to some references. 
Modern accounts were e.g. given by  Holst \cite{Ho,Ho2} and by
Camarri and  Pitman \cite{CP},\cite{Ca}. 
 Diaconis and Mosteller gave in \cite{DM} in conversational style a nice general discussion of studying ``coincidences". These works contain numerous references.\\ [0.1cm]
Similar questions are frequently studied and interpreted in the field of
``occupancy statistics", or using ``urn models".\\ 
In classical occupancy statistics one deals with the following situation:\\
$k$ distinguishable particles can be independently of each other in one of  $m$ states, there is no restriction on the number of
particles in a state, each allocation of the particles into the cells is equiprobable.\\
 It is clear from the above that the limit theorem for the r-repetition times are
simple consequences of Poisson limits for certain occupancy functionals. Therefore such limit theorems are  frequently implicit in certain poisson limit theorems in occupancy
statistics. 
The book \cite{Ko} by Kolchin, Sevastyanov and Chistyakov is authoritative 
on occupancy statistics (allocation of particles into cells).\\[0.1cm] 
For the application of urn models the book \cite{JK} by Johnson and Kotz is a comprehensive source.\\[0.1cm] 
Surprisingly, the special question of ``collision" times (as opposed to ``repetition" times) apparently has to the best of the authors knowledge not been dealt with directly in the statistical literature so far.\\[0.1cm]
There is a considerable amount of cryptographic literature on collision attacks. 
``Collision attacks" in their diverse forms (e.g. ``classical", ``meet in the middle" etc.) are a basic cryptanalytic tool and appeared already in the beginning of public
research on cryptography (e.g \cite{Co},\cite{Gir}). Also, from the beginning of hash function design, ``collision resistance" has been a primary goal.
However,  predictions for the effort of such attacks invariably used the ``random mapping assumption" (aka ``random oracle model").\\ A treatment of the classical birthday problem
for $r=2$ is given in many textbooks.\\
 Stinson\cite{Stin} analysed preimage-, 2nd -preimage and collision attacks in the random oracle model (using ``drawing without replacement"), and reductions between the diverse attacks.\\
Criticism of the random mapping assumption  was formulated  only relatively recently.\\
Bellare and Kohno\cite{BK}  were the first to point out that the random mapping assumption needs
justification. They discussed collision resistance for the cases of a fixed/ uniform random mapping and drawing with replacement (the model is explained in detail below), introduced a balance measure (essentially a variant
of the $\chi^2$-statistic, as we shall see below) and were mainly interested in bounds for the probabilities $\PP(K_2\leq k)$ of the collision time $K_2$ in terms of 
their balance measure. (They also treat bounds for the expectation of $K_2$ in an appendix). However, their bounds are not particularly tight. Much better bounds 
were obtained by Wiener \cite{Wi}, who considered ``drawing without replacement" and  treated the cases of uniform and ``multinomial" random functions (explained below)
as well as concrete functions.\\
Laccetti and Schmid \cite{LS} generalised Stinson's work to concrete mappings, gave exact expressions for the success probabilities of suitably randomised algorithms, and used
the theory of majorization to characterise the behaviour of the success probabilities in terms of the uniformity of the mapping.\\
Multi-collision attacks were considered
less often.\\
Preneel in his Ph.D. dissertation \cite{Pr} discussed the classical r-fold birthday problem, and essentially rederived independently some of von Mises' results. This case was
once more analysed by Suzuki, Tonien, Kurosawa and Toyota \cite{ST}, who showed that the $m^{(r-1)/r}$ effort is correct only for small fixed $r$ and that asymptotically
$\PP(R_r\leq (r!)^{1/r}m^{(r-1)/r})\geq \frac{1}{2}$ if $(r!/n)^{1/r}\approx 0$.\\
Nandi and  Stinson \cite{SN}
derive the effort for $r-$ collision attacks on uniform random mappings from an approximation given in \cite{DM}.\\
From the  practical side, Joux's\cite{Jo} spectacular multi-collision attack on cascaded constructions of iterated hash functions demonstrated
highly non-random behaviour in a class of hash-function constructions. This attack has been generalised in several ways.\\ 
 Finally, Ramanna and Sarkar \cite{RS} generalised Bellare and Kohno's 
work to the case of multi-collisions, along with some major improvements.\\
In all fairness it must be stated that almost all of these works plainly ignore most of the considerable statistical literature on the subject.
\section{Contributions of this work}
In contrast to the above mentioned works we give a treatment (apparently the first) of the collision times via generating functions and integral representations.
We consider all cases of interest (i.e. drawing with resp. without replacement and concrete as well as random mappings) and  demonstrate that using the classical apparatus
of occupancy statistics one can easily answer all questions of interest in a very precise manner. We obtain sharp estimates for the expectation of the collision resp.
repetition times, and for the first time obtain (under natural conditions) limit theorems for their distributions. \\
The results provide precise conditions under which the cryptographic folklore-beliefs are valid.

\section{Stochastic Model and First Orientation}
\subsection{The stochastic model}\label{ssec:stochmod}
Since the structure of the sets ${\mathcal D}$ resp. ${\mathcal R}$ is irrelevant for the generic collision search
we let w.l.o.g in the sequel ${\mathcal D}=\{1,\ldots,n\}$ and ${\mathcal R}=\{1,\ldots,m\}$.\\

The random production of preimages of $h$ can obviously be interpreted as ``drawing" from an urn which
contains resp. $x_i=|h^{-1}(\{i\})|$ balls of ``colour" $i$.\\

Consider the following situation:\begin{itemize}{ 
\item  an urn contains $n$ distinguishable balls of $m$ different colours,
namely $x_i$ (distinguishable) balls of colour $i$, $x_1+x_2+\ldots +x_m=n$
\item
balls are drawn (1) without replacement (2) with replacement from this urn, each draw costing
one time unit
\item in case (1) after each draw the next ball is chosen with uniform probability among the remaining balls, in case (2)
the next ball is chosen with uniform probability among all balls.}
\item the sampling is continued until the the random time point $K_r$ resp. $R_r$, at which for the first time $r$ different (resp.
$r$) balls of the same colour have been drawn
\end{itemize}
In the case of sampling with replacement it is necessary to distinguish between ``repetitions" (an image is hit $r$ times, but possibly from a repeated
 preimage)  and (true) multi-collisions.
\subsubsection*{Note on use of the word ``collision"}
In standard use of language a  second appearance during a drawing process (a duplicate) is interchangeably referred to a as a ``match", a ``coincidence", a ``collision" or a ``repetition".
Here we use ``collision" exclusively for the appearance of a ``true" collision, i.e. two (resp. $r$ pairwise) different domain points with the same image point,
and use ``repetition" for the second (resp. $r-$th) appearance of an image point.
\subsubsection*{Relation to the classical $r$-birthday problem}
It is well known  that the classical $r$-birthday problem may be formulated as an urn problem:\\ $k$ balls are drawn with replacement from an urn
containing $m$ different balls. How big is the chance that no ball is drawn more than $r-1$ times?\\
Thus case (2) with $n=m$, $x_1=\ldots=x_m=1$ - in the sequel called ``the classical case" - covers the classical $r$-birthday problem.
In this sense the repetition problems considered here are generalisations of the classical $r$-birthday problem. Note however that
in the classical setting no collisions (in the sense above)  are possible. 
\subsection{Generating functions for the occupancy numbers at time $k$}
\subsubsection{Fixed configuration}
Let us call the set $(x_1,\ldots,x_m)$ the ``configuration" of $h$ and let $Y_1(k),\ldots,Y_m(k)$ be the random variables  $Y_i(k):=$ ``number of {\bf different} balls of colour $i$" that have been drawn
(at time $k$).\\  Let further $Z_1(k),\ldots,Z_m(k)$ be the random variables  $Z_i(k):=$ ``number of balls of colour  $i$" that have been drawn (at time $k$).\\(In case 1  
 $Z_i(k)$  coincides with $Y_i(k)$ as only different balls are drawn when drawing without replacement).\\
The generating function for these variables can be derived by direct combinatorial arguments.\\
\begin{thm}
(1) In case 1 (drawing without replacement) the generating function of (the joint distribution of) 
$Y_1\eins(k),\ldots,Y_m\eins (k)$ is given by:
\begin{equation}
\EE\, t_1^{Y_1\eins(k)}\ldots t_m^{Y_m\eins(k)}=\tno (1+t\,t_1)^{x_1}\cdot\ldots\cdot(1+t\,t_m)^{x_m}
\end{equation}
(2a) In case 2 (drawing with replacement) the generating function of (the joint distribution of) 
$Y_1\zwei(k),\ldots,Y_m\zwei (k)$ is given by:
\begin{equation}
\EE\, t_1^{Y_1\zwei(k)}\ldots t_l^{Y_m\zwei(k)}=\tn2 (1+(e^t-1)\,t_m)^{x_1}\cdot\ldots\cdot(1+(e^t-1)\,t_m)^{x_m}
\end{equation}
(2b) In case 2 (drawing with replacement) the generating function of (the joint distribution of) 
$Z_1(k),\ldots,Z_m(k)$ is given by:
\begin{equation}
\EE\, t_1^{Z_1(k)}\ldots t_m^{Z_m(k)}=\tn2 e^{t t_1 x_1}\cdot\ldots e^{t t_m x_m}
\end{equation}
\end{thm}
\begin{proof}
(1) it is easily argued that 
\[\PP(Y_1\eins(k)=j_1,\ldots,Y_m\eins(k)=j_m)=\frac{{x_1 \choose j_1}\cdots {x_m \choose j_m}}{{n \choose k}}\]
(i.e. the joint distribution is the $m$-dim. hypergeometric distribution

with parameters $n,k$ and $x_1,\ldots,x_m$). The generating function follows.\\[0.1cm]
(2a) we have 
\[\PP(Y_1\zwei(k)=j_1,\ldots,Y_m\zwei(k)=j_m)=\frac{1}{n^k}\cdot {x_1 \choose j_1}\cdots {x_m \choose j_m}\cdot Sur(k,j_1+\ldots+j_m)\]
where $Sur(k,r)$ denotes the number of surjective mappings from $\{1,\ldots,k\}$ onto $\{1,\ldots,r\}$.\\ It is known that
$Sur(k,r)=k!\,[t^k]\,(e^t-1)^r$ (since a such a surjective mapping corresponds uniquely to an ordered partition of
$\{1,\ldots,k\}$ in $r$ into non-empty subsets, and  $e^t-1$ is the exponential generating function for non-empty sets). The generating
function follows. \\[0.1cm]

(2b) the generating function follows from the  fact that
\[\PP(Z_1(k)=j_1,\ldots,Z_m(k)=j_m)=\frac{1}{n^k}\cdot {k \choose j_1,\ldots,j_m} x_1^{j_1}\cdots x_m^{j_m}\]
(i.e. the joint distribution of $(Z_1(k),\ldots,Z_m(k))$ is the multinomial distribution with parameters $k$ and $p_1:=\frac{x_1}{n},\ldots, p_m:=\frac{x_m}{n}$.)
\end{proof}

\subsubsection{Random configuration}
Now let us now consider the following two-stage random experiment:\
\begin{enumerate}
\item in the first stage the urn is filled at random with the balls of different colours
\item in the second stage the balls are drawn from the urn
\end{enumerate}
i.e. the numbers $x_i$ are realisations of random variables $X_i$, where $X_1+\ldots X_m=n$, and conditional
on $X_1=x_1,\ldots,X_m=x_m$ the situation above is given.\\[0.1cm]
Let  $g_{(X_1,\ldots,X_m)}(t_1,\ldots,t_m):=\EE\,t_1^{X_1}\ldots t_m^{X_m}$ denote the generating function of
$(X_1,\ldots,X_m)$ and let $Y_1(k),\ldots,Y_m(k)$ and 
$Z_1(k),\ldots,Z_m(k)$ have the same meaning as above, but for the two-stage experiment..
Then the following is immediate:
\begin{cor} If the urn is filled with a random configuration $(X_1,\ldots,X_m)$ the following
statements hold\\
(1) In case 1 (drawing without replacement) the generating function of (the joint distribution of) 
$Y_1\eins(k),\ldots,Y_m\eins (k)$ is given by:
\begin{equation}
\EE\, t_1^{Y_1\eins(k)}\ldots t_m^{Y_m\eins(k)}=\tno g_{(X_1,\ldots,X_m)}(1+t\,t_1,\ldots,1+t\,t_m)
\end{equation}
(2a) In case 2 (drawing with replacement) the generating function of (the joint distribution of) 
$Y_1\zwei(k),\ldots,Y_m\zwei (k)$ is given by:
\[\EE\, t_1^{Y_1\zwei (k)}\ldots t_m^{Y_m\zwei (k)}=\tn2 g_{(X_1,\ldots,X_m)}(1+(e^t-1)\,t_1,\ldots,1+(e^t-1)\,t_m)\]
(2b) In case 2 (drawing with replacement) the generating function of (the joint distribution of) 
$Z_1(k),\ldots,Z_m(k)$ is given by:
\[\EE\, t_1^{Z_1(k)}\ldots t_m^{Z_m(k)}=\tn2 g_{(X_1,\ldots,X_m)}(e^{tt_1},\ldots,e^{tt_m})\]
\end{cor}

For a uniform random $(n,m)$-mapping the joint distribution of $(X_1,\ldots,X_m)$ is the multinomial distribution
with parameters $n$ and $p_1=\ldots,p_m=\frac{1}{m}$. For a random $(n,m)$-mapping where the images take independently
the value $i$ with probability $p_i$ it is the multinomial distribution with parameters $n$ and $p_1,\ldots,p_m$.
Mappings with multinomially distributed preimage sizes $(X_1,\ldots,X_m)$ will in the sequel be called ``multinomial".\\
Let us consider this case as an example.

\begin{exl}\label{ex1}
Let $p_1,\ldots,p_m\geq 0$ with $p_1+\ldots +p_m=1$ and let $$g_{(X_1,\ldots,X_m)}(t_1,\ldots,t_m)=(t_1p_1+\ldots +t_m p_m)^n.$$
Then for $k\leq n$ 
\[\EE\, t_1^{Y_1\eins(k)}\ldots t_m^{Y_m\eins(k)}=\tno (1+t(t_1p_1+\ldots t_m p_m))^n=(t_1p_1+\ldots +t_m p_m)^k\]
That is, the joint distribution of $(Y_1\eins(k),\ldots,Y_m\eins(k))$ is multinomial with parameters $k$ and $p_1,\ldots,p_m$. 
Note that for $p_1=\ldots =p_m$ this is exactly the situation of 
the classical birthday problem (for $m$ birthdays).
{\bf Thus (as long as $k\leq n$) drawing without replacement from a (uniform) multinomial configuration leads to the same occupancy 
distribution at time $k$ as in the setting of the classical birthday problem in the codomain.}
\end{exl}
In other words:\\
sampling $k$ times without replacement from a multinomial $n,(p_1,\ldots,p_m)$ mapping produces a multinomial 
$k,(p_1,\ldots,p_m)$ mapping, 
and sampling without replacement from a $n,(\frac{x_1}{n},\ldots,\frac{x_m}{n})$ multinomial random $(n,m)$-mapping produces for $k\leq n$ the same distribution
of $(Z_1(k),\ldots,Z_m(k))$
as does sampling with replacement from a fixed $(n,m)$ mapping with configuration $(x_1,\ldots,x_m)$).\\ In these cases the 
the description by a multinomial random mapping is exact!

\subsection{Notation and conventions}
Let w.l.o.g. all random variables appearing in the sequel be defined on the same probability space  $(\Omega,{\mathcal A},\PP)$.
The cases ``drawing without"  resp. ``with replacement" are called ``case 1" resp. ``case 2" in the sequel and indicated by the use of the superscripts ``$1$" resp. ``$2$".\\
The collision degree $r$ is a fixed (small) number $r\geq 2$, $r=2$ is the most interesting case.\\

We use the convention $\inf(\emptyset)=\infty$ and let 
$K_r\eins,K_r\zwei$ resp. $R_r$ : $\Omega \lra \mathbb{N}\cup\{\infty\}$ be the random variables

\[\begin{array}{rcl}K_r\eins(\omega)&:=& \inf \{ k\geq 1\;:\; \exists\, i\in\{1,\ldots,m\}\mbox { s.th. } Y_i\eins(k)(\omega)\geq r\}\\[0.2cm]

K_r\zwei(\omega)&:=& \inf \{ k\geq 1\;:\; \exists\, i\in\{1,\ldots,m\}\mbox { s.th. } Y_i\zwei(k)(\omega)\geq r\}\end{array}\] 
``time of the first $r-$-collision" when drawing without resp. with replacement
\[R_r(\omega):= \inf \{ k\geq 1\;:\; \exists\, i\in\{1,\ldots,m\}\mbox { s.th. } Z_i(k)(\omega)\geq r\}\] 
``time of the first $r$-fold repetition" when drawing with replacement.\\

In case 1 (drawing without replacement) it does not make sense to draw from an empty urn. Therefore we stipulate for this case that the drawing process is stopped latest at the
$(n+1)$-drawing (i.e. after the first observation that the urn is empty.). $K_r\eins(\omega)>n$ then means that no $r$-collisions have appeared.

To guarantee the finiteness of these waiting times we assume in the sequel that at least for one $i$ we have $x_i\geq r$ resp. $\PP(X_i\geq r)>0$.\\

It is clear that the distributions of $K_r\eins,K_r\zwei$ resp. $R_r$ depend on the parameters $n,m,x=(x_1,\ldots,x_m)$ resp. $n,m, \PP_{(X_1,\ldots,X_m})$,
but this dependency is for convenience suppressed from the notation.

\subsection{First orientation: expected number of $r$-collisions at time $k$}

For a first orientation one will compute the expected number of $r-$collisions at time $k$. With $Y_i(k)$ resp. $Z_i(k)$ as above 
the number of $r$-collisions is given
by the random variables
$$S_{r}\eins(k):=\sum_{i=1}^{m} {Y_{i}\eins(k) \choose r}\mbox{ resp.  }  S_{r}\zwei(k):=\sum_{i=1}^{m} {Y_{i}\zwei(k) \choose r}$$
while for sampling with repetition the number of $r-$ multi-sets with colliding image (but possibly repeated preimages)
is given by
$$C_r(k):=\sum_{i=1}^m {Z_i(k) \choose r}$$
Let $s_{r,n}:=\sum_{i=1}^m { x_i \choose r}$.
Using the generating functions given above we find:
$$\EE(S_{r}\eins(k))=\frac{(k)_r}{(n)_r}\, s_{r,n}$$
$$\EE(C_r(k)):=\frac{(k)_r}{r!\,n^r} \sum_{i=1}^m x_i^r$$
$$\EE(S_{r})\zwei(k)=\left(\sum_{i=0}^r {r \choose i} (-1)^i (1-\frac{i}{n})^k\right)s_{r,n}
=\left(\frac{(k)_r}{n^r}(1-\frac{r(k-r)}{2n}+...)\right)s_{r,n} $$
Remark: in the case of sampling with repetition we count here an occurring  $r$-collision only once. The corresponding
$r$-set may of course repeatedly occur in sequence of preimages. For the case of counting with multiplicities one gets \cite{RS}:
$$\EE(S_{r,2,multi})(k)=\frac{(k)_r}{n^r}\, s_{r,n}$$
 
Thus we have the following picture:\\[0.1cm]
(1) first $r$ collisions resp. $r$-repetitions will appear at times of magnitude $t_C:=n/(s_{r,n})^{1/r})$ resp. of magnitude 
$t_R:=n/(\sum_{i=1}^m \frac{x_i^r}{r!})^{\frac{1}{r}}$.\\[0.1cm]
(2) one will hope for limit theorems if the the parameters
$n,m,x$ are varied in such a way that for  $t=t_C$ (resp. $t=t_R$) $t=t(n,m,x)\lra \infty$.\\ If the cell probabilities $p_i=\frac{x_i}{n}$ are uniformly small one will
expect (on the ground of known limit theorems for weakly dependent indicator variables) expect Poisson limits for the number of $r-$collisions resp.
$r$-repetitions
.\\[0.1cm]
(3) the difference between drawing with. resp. without replacement will only be notable if $n$ is relatively small or if $\sum_{i=1}^m p_i^r$ 
is of magnitude $\frac{1}{n^{r-1}}$.\\[0.1cm]
(4) in the case of sampling with replacement the $r$-repetitions will appear no later as times of magnitude  $n^{\frac{r-1}{r}}$ (as by \ref{thmvM} resp. \ref{corvM} $r$-repetitions
of preimages will appear at this time).\\

\section{Exact treatment of the collision times}

For the formulation of generating functions we need some definitions.
For $n\in \mathbb{N}, t\in \mathbb{R}$ 
let \begin{eqnarray} 
 p_r(n,t)&:=& \sum_{i=0}^{r-1} {n \choose i} t^i\\
 q_r(n,t)&:=&\sum_{i=0}^{r-1} \frac{n^it^i}{i!}\\
 G_r(n,t)&:=&\sum_{i=0}^{r-1} {n \choose i} t^i(1-t)^{n-i}\end{eqnarray}
Note that for $0\leq t \leq 1$  the function $F_r(n,t):=1-G_r(n,t)$ is the distribution function of the
$r$-th largest element (the $r$-th order statistic) of a sample of $n$ i.i.d. variables uniform on $[0,1]$.

\subsection{ Combinatorial formulae}

We clearly have:
\begin{rem}
\begin{eqnarray*}
\PP(K_r>k)&=&\PP(Y_1(k)\leq r-1,\ldots,Y_m(k)\leq r-1)\;\;\;\mbox{ and }\\\PP(R_r>k)&=&\PP(Z_1(k)\leq r-1,\ldots,Z_m(k)\leq r-1)
\end{eqnarray*}
\end{rem}
\subsubsection{Fixed mapping}
Let again be  $h:\{1,\ldots,n\}\lra\{1,\ldots,m\}$ a mapping with preimage cardinalities $x_i=|h^{-1}(\{i\})|$.\\
From the generating functions of the cell occupancies given above  we get exact combinatorial expressions
for the probabilities in question.

\begin{thm}\label{thmprgen}
\begin{eqnarray}
\PP(K_{r}\eins>k)&=&\tno \prod_{i=1}^m p_r(x_i,t)\\ 
\PP(K_{r}\zwei>k)&=&\tn2 \prod_{i=1}^m p_r(x_i,e^t-1)\\ 
\PP(R_r>k)&=&\tn2 \prod_{i=1}^m q_r(x_i,t)
\end{eqnarray}
\end{thm}

If we denote the $k-$th elementary symmetric function of the variables $x_1,\ldots,x_m$ by $Sym_k(x_m,\ldots,x_m)$ and
the number of surjective mappings of a $k$-element set onto a $d$ element set by $Sur_(k,d)$ we thus have 
in particular for $r=2$
\begin{cor}
For $r=2$

\begin{eqnarray}
\PP(K_{2}{\eins} > k)&=&\frac{k!\,(n-k)!}{n!}\,Sym_k(x_1,\ldots,x_m)\\
\PP(K_{2}{\zwei} > k)&=&\frac{1}{n^k}\sum_{d=0}^k\,Sur(k,d)\,Sym_d(x_1,\ldots,x_m)\\
\PP(R_{2} > k)&=&\frac{k!}{n^k}\,Sym_k(x_1,\ldots,x_m)
\end{eqnarray}

\end{cor}

Since $x_1+\ldots +x_m=n$, and since we require that at least one $x_i$ be $\geq r$, the products expressing the different generating fucntions above give polynomials in $t$ with degree $\leq n-1$. In particular, in case (1)
$\PP(K_r>n)=0$ and in case (2)  $\PP(R_r>n)=0$.\\

The explicit form of the generating functions makes it possible to prove some ``intuitively obvious" properties. We consider two such
intuitions.\\

 Firstly one expects that $r$-collision search becomes harder the smaller the individual preimage sizes are. The next lemma 
shows this intuition to be true in a very strong sense.
 
\begin{lem}\label{majlem}
 (stochastic ordering over configurations)\begin{enumerate} \item if in a configuration $(x_1,\ldots,x_m)$ there are images  $i,j$ s.th. $r\leq x_i<x_j-1$, then the probabilities
 $\PP(K_r>k)$ can only increase if
$x_i$ is replaced by $x_i+1$ and $x_j$ is replaced by $x_j-1$.
\item if in a configuration $(x_1,\ldots,x_m)$ there are images $i,j$ s.th. $0\leq x_i<x_j-1$, then the probability $\PP(R_r>k)$ can only increase if
$x_i$ is replaced by $x_i+1$ and $x_j$ is replaced by $x_j-1$, or if each is replaced by $(x_i+x_j)/2$.
\end{enumerate}
\end{lem}
\begin{proof}
(1) let $x,y\in\mathbb{N},\,r\leq x\leq y-1$. Using the recursion $p_r(n,t)=p_{r}(n-1,t)+tp_{r-1}(n-1,t)$ it is not hard to show that
\[p_r(x,t)p_r(y,t)-p_r(x+1,t)p_r(y-1,t)=t^r\left({x \choose r-1} p_{r-1}(y-1,t)-{y-1 \choose r-1} p_{r-1}(x,t)\right)\]
The non-zero coefficients of $t$ on the rhs are ${y-1 \choose j}{x \choose r-1}-{x \choose j}{y-1 \choose (r-1)},\,0\leq j\leq (r-1)$
are thus nonnegative.\\
(2) let $x,y\in\mathbb{R}_+,x<y$. We show that $ [t^k] q_r(xt)q_r(yt)\leq  [t^k] \left(q_r(((x+y)/2)t\right)^2$.\\ It is easy to see that
equality holds for $k\leq r-1$ and $k>2r-2$. Let $r\leq k \leq 2r-2$.  We have 
$$k!\,[t^k] q_{r}(xt)q_r(yt)=(x+y)^k-\sum_{i=0}^{k-r} {k \choose i}\left(x^i y^{k-i}+y^i x^{k-i}\right)
$$
For fixed sum $s=x+y$  the function $x \mapsto  f(x):=\sum_{i=0}^{k-r} {k \choose i}\left(x^i y^{k-i}+y^i x^{k-i}\right)$ has the derivative
$f^\prime(x)=(k-r+1){k \choose r} \left(x^{k-r+1}(s-x)^r-(s-x)^{k-r+1}x^r\right)$. Thus $x\mapsto f(x)$ 
is strictly decreasing (resp. increasing) on $[0,s/2]$ (resp. [s/2,s]), attaining its minimum at $x=s/2$.
\end{proof} 
The $r$-collision times are therefore stochastically largest when the preimage sizes $\geq r$  are as uniform as possible,
and the $r$-repetition time is stochastically largest when the preimage sizes are as uniform as possible.\\

Secondly one expects that collisions are easier to find when it is guaranteed that the sampled domain points
are mutually different, i.e. if sampling without replacement is used. The following lemma shows that this is indeed
true.

\begin{lem} (stochastic ordering between the different waiting times)\\
1. For $k\in \mathbb{N}$ and $r\geq 2$
\[\PP(K_r\zwei>k)\geq \PP(K_r\eins>k) \mbox{ and } \PP(K_r\zwei>k)\geq P(R_r>k)\]
2. For $r=2, k\in \mathbb{N}$ also \[ \PP(K_2\eins>k)\geq \PP(R_2>k)\]
\end{lem}
\begin{proof}
1. by the formulae in theorem 4.2 above 
\[\hspace*{-0.8cm}\PP(K\zwei_r>k)=\frac{1}{n^k}\sum_{d=0}^{k} Sur(k,d){n \choose d} \PP(K\eins_r >d)\geq \PP(K\eins_r >k)\left(\frac{1}{n^k}\sum_{d=0}^{k}  Sur(k,d) {n \choose d}\right)=\PP(K\eins_r>k)\]
where we have used that $\sum_{d=0}^{k} Sur(k,d){n \choose d}=n^k$.\\
2.:  since $p_2(x,t)=1+xt=q_2(xt)$ the formulae above give
$$\PP(R_2 > k)=\frac{n!}{(n-k)!n!} \PP(K_2\eins >k)$$
Thus $R_2$ is distributed as the $\min(T,K_2\eins)$ where $T$ is independent of $K_2\eins$, and distributed as the waiting time for the
first $2$-collision of preimages.

\end{proof}

In case 2 the numbers ${n \choose d}\frac{Sur_(k,d)}{n^d}=:\PP(I^{(k,n)}=d)$ give the probabilities that the image 
of the (uniform) random mapping $\{1,\ldots,k\}\lra {\mathcal D}$ given by $i\mapsto D_i$ has cardinality $d$. Using this
r.v. $I^{(k,n)}$ we have $\PP(K\zwei >k\,|\, I^{k,n}=d)=\PP(K\eins >d)$, and 
$$\PP(K\zwei>k)=\sum_{d=0}^k \PP(K\eins>k)\PP(I^{(k,n)}=d)$$
These relations hold for any fixed configuration, and thus also for random configurations.\\

If we denote ``$X$ is stochastically  larger than $Y$" by $Y\preceq X$ we may summarise lemma 2 as:
For $r=2$ we have $R_2\preceq K\eins \preceq K\zwei$, and for $r\geq 3$ we have $R_r\preceq K_r\zwei$ and 
$K_r\eins\preceq K_r\zwei$.
The following example shows that in general for $r\geq 3$ no stochastic ordering between $K_r\eins$ and $R_r$ exists:

\begin{exl}
Let $r\geq 3, n=r+1, m=2, x_1=1, x_2=r$. Then on the one hand $$\PP(K_r\eins=r)=\frac{1}{r+1}<\frac{1+r^r}{(r+1)^r}=\PP(R_r=r)$$
and therefore $\PP(K\eins>r)>\PP(R_r>r)$. On the other hand $\PP(K_r\eins>r+1)=0$ but $\PP(R_r>r+1)\geq \PP(R_r>2r-2)>0$.
\end{exl}

\subsubsection{Random configurations}

For multinomial $(n,m)$-random mappings (with parameters $n$ and $p_1,\dots,p_m$) we find:

\begin{thm}
\begin{eqnarray}\mbox{(for $k\leq n$)  }\; \PP(K_r^{(1)}>k)&=& k![t^k] \prod_{i=1}^m q_r(p_i t) \\
\hspace*{1cm} \PP(K_r\zwei>k)&=&\sum_{d=0}^k \PP(K_r\eins >d)\,\PP(I^{(k,n)}=d)
\end{eqnarray}
\end{thm}
\begin{proof} From \autoref{ex1} we know that in case 1 the joint distribution of $(Y\eins_1(k),\ldots,Y\eins_m(k))$
is for $k\leq n$ the multinomial distribution with parameters $k$ and $p_1,\ldots,p_m$. Thus
$$\EE\, t_1^{Y_1\eins(k)}\ldots t_m^{Y_m\eins(k)}=(t_1p_1+\ldots +t_m p_m)^k=k!\,
[t^k]\prod_{i=1}^m e^{p_it_i t }$$
The representation above follows.
\end{proof}
Again we note the case $r=2$ separately:
\begin{cor}
\begin{eqnarray} \PP(K_2^{(1)}>k)&=& k!\,Sym_k(p_1,\ldots,p_m) \\
\PP(K_2\zwei>k)&=&\frac{1}{n^k}\sum_{d=0}^k\,Sur(k,d)\,\frac{n!}{(n-d)!} Sym_d(p_1,\ldots,p_m)\\
\PP(R_2>k)&=&\frac{k!}{n^k}\frac{n!}{(n-k)!}\,Sym_k(p_1,\ldots,p_m)
\end{eqnarray}
\end{cor}

Some remarks are in order:
\begin{enumerate}
\item $r$-collisions can only occur if a cell ocupancy $x_i\geq r$ exists, and the appearance of the terms ${x_i \choose r}$ in the
formulae for the collision times (for concrete mappings) reflects this (the formulae for the repetition times contain the terms $x_i^r$ instead).
(This was earlier remarked by Wiener \cite{Wi} and Ramanna and Sarkar \cite{RS}).
\item the consideration of random configurations involves an ``averaging" over an ensemble of mappings. It is then clear that the description
by a random mapping is not appropriate for ``improbable" mappings. E.g. if $n=m$ and $h$ is a permutation no collisions are possible. It turns out,
however, that the characteristics which determine the collision behaviour have under certain conditions a weak limit for ``very large" random mappings (i.e.
have values of the same order of magnitude  for ``almost all" mappings). In these cases we may say that random mappings have a typical collision behaviour.  
\item probability bounds for the collision times: 
since we are
interested in limit theorems rather than probability bounds we shall not pursue this here. But we note that  
the concrete representation of the probabilities as coefficients of power series makes it possible to use the
saddle point method to derive very good estimates for these probabilities (in the same way as Good \cite{Go} did for the multinomial distribution). 
\end{enumerate}

\subsection{The generating functions for the collision/repetition times}

There are several methods to treat the distributions of these waiting times. We use (integral representations) for the
generating functions of the probabilities $\PP(K_r>k)$ resp. $R_r>k$, because these are particularly near at hand, given the
explicit representation for the coefficients.

Let in the sequel 
\[g_r\eins(u):=\sum_{k=0}^{\infty} u^k \PP(K_r\eins >k)\] 
\[g_r\zwei(u):=\sum_{k=0}^{\infty} u^k \PP(K_r\zwei >k)\] 
\[h_r(u):=\sum_{k=0}^{\infty} u^k \PP(R_r>k)\] 

(Recall the convention $\PP(K\eins>n+1)=0$.) 

These series are clearly convergent for $|u|<1$, and we have
\[g_r\eins(1)=\EE(K_r\eins),\;\;\; g_r\zwei(1)=\EE(K_r\zwei),\;\;\; h_r(1)=\EE(R_r).\]

\begin{thm}\label{thmgenrep}

\begin{eqnarray} g_r\eins(u)&=&(n+1)\int_0^1 (1-t)^n \prod_{i=1}^m p_r(x_i,\frac{ut}{1-t})\,dt \\ &=& (n+1) \int_0^{\infty} e^{-(n+1)s}\,\prod_{i=1}^m p_r(x_i,u(e^s-1))\;ds\\
g_r\zwei(u)&=&n \int_0^{\infty} e^{-ns}\,\prod_{i=1}^m p_r(x_i, e^{us}-1)\;ds\\
h_r(u)&=&n \int_0^{\infty} e^{-ns}\,\prod_{i=1}^m q_r(x_ius)\;ds\end{eqnarray}
\end{thm}
\begin{proof}
Using the well known relations $\int_{0}^1 t^{a}(1-t)^{b}\,dt=\frac{a!b!}{(a+b+1)!}$ and $\int_0^\infty s^k e^{-s}=k!$ this follows imediately from
the representation in \autoref{thmprgen} above. Note that summation and integration may be freely interchanged, since all sumands are nonnnegative.
\end{proof}

\subsection{Expectation of the waiting times}

For the expectations we have thus the following expressions 

\begin{prp}
\begin{eqnarray}
\EE(K_r\eins)&=&(n+1)\int_0^1 (1-t)^n \prod_{i=1}^m p_r(x_i,\frac{t}{1-t})\,dt\\&=& (n+1)\int_0^{\infty} e^{-(n+1)s}\,\prod_{i=1}^m p_r(x_i,(e^s-1))\;ds\\
\EE(K_r\zwei)&=&n \int_0^{\infty} e^{-ns}\,\prod_{i=1}^m p_r(x_i, e^{s}-1)\;ds\\
\label{erwdar}\EE(R_r)&=&n \int_0^{\infty} e^{-ns}\,\prod_{i=1}^m q_r(x_is)\;ds\end{eqnarray}
\end{prp}

Some remarks:
\begin{enumerate}
\item As $p_r(x,t)=(1+t)^x $ for $x\leq r-1$ we may rewrite the expectations of the collision
times using $M_r:=\{i\,:\,x_i\geq r \}$ as follows
$$\EE(K_r\eins)=(n+1)\int_0^1 \prod_{i\in M_r} G_r(x_i,t)\,dt$$
$$\EE(K_r\zwei)=n \int_0^{\infty} \prod_{i\in M_r} G_r(x_i,1-e^{-t})\,dt$$
The quotients $\EE(K_r\eins)/(n+1)$ and $\EE(K_r\zwei/n)$ thus depend only on the $x_i$ with $x_i\geq r$. In contrast
$\EE(R_r)/n$ depends on all $x_i$ with $x_i>0$. 
\item if $T_1,\ldots,T_m$ are independent $B(r,x_i+1-r)$-variables we have 
$\EE(K_r\eins)=(n+1)\,\EE(\min\{T_1,\ldots,T_m\})$. Similarly, if $T_1,\ldots,T_m$ are independent, and $T_i$ is distributed as the $r$-th order statistic of 
$x_i$ independent exponential variables with mean 1, we have
$\EE(K_r\zwei)=n\,\EE\min\{T_1,\ldots,T_m\}$.\\ Finally, if  $T_1,\ldots,T_m$ are independent $\Gamma(x_i,r)$ distributed we have
$\EE(R_r)=n\,\EE\min\{T_1,\ldots,T_m\}$ (this was earlier shown by Holst \cite{Ho2}).
\item for the classical case the representation \eqref{erwdar} is a well known result due to Klamkin and Newman (\cite{KM}). The general multinomial
case was given in \cite{GTF} and \cite{Ho2}.

\end{enumerate}

In simple cases the integrals can be evaluated explicitly. In the sequel $B(a,b)$ denotes the Beta function.
\begin{exl} 
Let $x_i\leq r$ for all $i$, and $a:=|\{ i\,:\,x_i=r\}|\geq 1$. Then
\begin{eqnarray*}
\EE(K_r\eins)&=&(n+1)\int_0^1 (1-t^r)^a\,dt=\tfrac{n+1}{r} B(\tfrac{1}{r},1+a)=(n+1)\tfrac{a!\,r^a}{\prod_{i=1}^a (1+ir)}\approx \Gamma(1+\tfrac{1}{r})\tfrac{n+1}{(1+a)^{1/r}}\\
\EE(K_r\zwei)&=& n\int_0^1 (1-t^r)^a\tfrac{1}{1-t}\,dt=\tfrac{n}{r}\,\left(\sum_{i=1}^{r-1} B(\tfrac{i}{r},a+1)\right)\end{eqnarray*}
In the case $a=1$ of only one possible $r$-collision thus
$$\EE(K_r\eins)=\tfrac{r}{r+1}\;(n+1)\mbox{  and } \EE(K_r\zwei)=n\,(1+\tfrac{1}{2}+\ldots +\tfrac{1}{r})$$
Note that this example includes the exact solutions for the case of ``$r$-regular" functions, where $r$-regular means that $a=m,\,n=m r$.
\end{exl}

\begin{exl}
Let $x_i\geq r$ for only one $i$, $x_i=x$. Then 
$$\EE(K_r\eins)=r\;\tfrac{n+1}{x+1}\;\;\mbox{  and  }\;\;\EE(K_r\zwei)=n\;\sum_{i=0}^{r-1}\tfrac{1}{x-i}$$
In particcular, in the case of a constant mapping ($x=n$) thus
$$\EE(K_r\eins)=r\;\;\mbox{  and  }\;\;\EE(K_r\zwei)=r+\sum_{i=1}^{r-1}\tfrac{i}{n-i}$$
\end{exl}

\begin{prp}(asymptotic expansion in the classical case)\\
Let $n=m$, $x_1=x_2=\ldots=x_m=1$. Then $\EE(R_r)/n$ has an asymptotic expansion in powers of $n^{-1/r}$.
\end{prp}
\begin{proof}
(Sketch) By Klamkin's and Newman's formula
$$\EE(R_r)=n\,\int_0^{\infty} (q_r(y)\,e^{-y})^n\,dy=:n\, I_r(n)$$
The substitution $t=\left(r!(y-\log(q_r(y))\right)^{{1/r}}$ transforms $I_r(n)$ to
$$I_r(n)=\int_{0}^\infty e^{-n\tfrac{t^{r}}{r!}}\,\tfrac{t^{r-1}}{y(t)^{r-1}}\,q_r(y(t))\,dt$$
Since $y(t)>t$ for $t>0$ (comp. \autoref{lema3}) and since $t\mapsto q_r(t)/t^{r-1}$ is decreasing
$$I_r(n)<\int_{0}^\infty e^{-n\tfrac{t^{r}}{r!}}\,\,q_r(t))\,dt=\tfrac{1}{r}\sum_{i=0}^{r-1}\tfrac{1}{i!}(\tfrac{r!}{n})^{(i+1)/r}\Gamma(\tfrac{i+1}{r})$$
Further, $t\mapsto \tfrac{t^{r-1}}{y(t)^{r-1}}\,q_r(y(t))=: g_r(t)$ is analytic in a neighbourhood of $0$, say $g_r(t)=\sum_{i=0}^\infty {a_i(r)} t^i$.
An application of Laplace's method (e.g. \cite{deB}, chap. 4) now gives
$$I_r\sim \frac{1}{r} \sum_{i=0}^\infty a_i(r) \Gamma(\tfrac{i+1}{r})\left(\tfrac{r!}{n}\right)^{(i+1)/r}$$
\end{proof}
For instance, in the classical case 
\begin{eqnarray*}
\hspace*{-2.5cm}\EE(R_2)&=&\tfrac{n}{2}\,\left((\tfrac{2}{n})^{1/2}\Gamma(\tfrac{1}{2})+\tfrac{2}{3}\,(\tfrac{2}{n})^{1/1}\Gamma(\tfrac{2}{2})+\tfrac{1}{12}(\tfrac{2}{n})^{3/2}\Gamma(\tfrac{3}{2})-\tfrac{2}{135}(\tfrac{2}{n})^{4/2}\Gamma(\tfrac{4}{2})+\tfrac{1}{864}(\tfrac{2}{n})^{5/2}\Gamma(\tfrac{5}{2})+\ldots\right)\\
\hspace*{-2.5cm}\EE(R_3)&=&\tfrac{n}{3}\,\left((\tfrac{6}{n})^{1/3}\Gamma(\tfrac{1}{3})+\tfrac{1}{2}\,(\tfrac{6}{n})^{2/3}\Gamma(\tfrac{2}{3})+\tfrac{21}{80}(\tfrac{6}{n})^{3/3}\Gamma(\tfrac{3}{3})+\tfrac{7}{240}(\tfrac{6}{n})^{4/3}\Gamma(\tfrac{4}{3})+\tfrac{83}{13440}(\tfrac{6}{n})^{5/3}\Gamma(\tfrac{5}{3})+\ldots\right)\\
\hspace*{-2.5cm}\EE(R_4)&=&\tfrac{n}{4}\,\left((\tfrac{24}{n})^{1/4}\Gamma(\tfrac{1}{4})+\tfrac{2}{5}\,(\tfrac{24}{n})^{2/4}\Gamma(\tfrac{2}{4})+\tfrac{17}{100}(\tfrac{24}{n})^{3/4}\Gamma(\tfrac{3}{4})+\tfrac{194}{2625}(\tfrac{24}{n})^{4/4}\Gamma(\tfrac{4}{4})+\tfrac{271}{42000}(\tfrac{24}{n})^{5/4}\Gamma(\tfrac{5}{4})+\ldots\right)\\
\hspace*{-2.5cm}\EE(R_5)&=&\tfrac{n}{5}\,\left((\tfrac{120}{n})^{1/5}\Gamma(\tfrac{1}{5})+\tfrac{1}{3}\,(\tfrac{120}{n})^{2/5}\Gamma(\tfrac{2}{5})+\tfrac{5}{42}(\tfrac{120}{n})^{3/5}\Gamma(\tfrac{3}{5})+\tfrac{11}{252}(\tfrac{120}{n})^{4/5}\Gamma(\tfrac{4}{5})+\tfrac{515}{31752}(\tfrac{120}{n})^{5/5}\Gamma(\tfrac{5}{5})+\ldots\right)
\end{eqnarray*}

The case $r=2$ of this proposition is well known (cmp. \cite{Kn}, section 1.2.11.3  and problem 20 there), note that in the classical case $\EE(R_2)=1+Q(n)$), and for $r=3$ the first three terms were given by Holst \cite{Ho2}, but I couldn't find the asymptotic series for 
higher $r$ in the literature. (Klamkin and Newman give only the first term).  Since 
even the encyclopaedic work \cite{FlSedg} (where the problem is treated on p. 116)  doesn't mention them, they may at least not be well known. To make the remainder $o(1)$ one has to take $r$ terms of the asymptotic series for $\EE(R_r)$. For example, for $r=3$ and $n=365$
the first term gives $\EE(R_3)\approx 82,87442$, the first three terms give $\EE(R_3)\approx 88,72504$ while the exact value is $\EE(R_3)=88,73891...$\\

 Let in the sequel $\tilde{s}_r:=\tfrac{\sum_{i=1}^m x_i^r}{r!}$,\,$s_r:=\sum_{i=1}^m{x_i \choose r}$
and let w.l.o.g $x_1:=\max\{x_1,\ldots,x_m\}$. 
We find the following bounds for the expectations:

\begin{thm}\label{thml}(lower bounds)
\begin{eqnarray}
\label{eq1} \EE(K_r\eins)&\geq& (n+1)\, B(\tfrac{1}{r},1+s_r)>\Gamma(1+\tfrac{1}{r})\,\tfrac{n+1}{({s_r+1})^{1/r}}\\
\label{eq2} \EE(K_r\zwei)&>&\Gamma(1+\tfrac{1}{r})\,\tfrac{n}{(s_r)^{1/r}}\\
\label{eq3} \EE(R_r)&>&\Gamma(1+\tfrac{1}{r})\,\tfrac{n}{(\tilde{s}_r)^{1/r}}
\end{eqnarray}
\end{thm}
\begin{proof} 
\ref{eq1} The left inequality  follows using the inequality $G_r(x,t)\geq (1-t^r)^{x \choose r}$ (cmp. \autoref{lema1}). The second inequality follows
from Jensen's inequality: let $X$ be a $\Gamma(s+1,1)$-distributed random variable, then 
$$\tfrac{\Gamma(s+1+\tfrac{1}{r})}{\Gamma(s+1)}=\EE(\sqrt[r]{X})\leq \sqrt[r]{\EE(X)}=\sqrt[r]{s+1}$$
\ref{eq2} follows from $G_r(x,1-e^{-s})\geq e^{-{x \choose r}s^r}$  (for $s>0,x\geq r$, cmp.\autoref{lema2})\\
\ref{eq3} follows from $q_r(t)\,e^{-t}> e^{-\tfrac{t^r}{r!}}$ (for $t>0$, cmp. \autoref{lema3})\\
\end{proof}

Our next aim is to find upper bounds. Let $u_r:=\tfrac{\sum_{i\in M_r} x_i}{|M_r|}$ and $w:=\min\{n,m\}$. Directly from the majorisation \autoref{majlem} we have
\begin{prp} (upper bounds 1)
\begin{eqnarray}
\EE(K_r\eins)&\leq& (n+1)\, B(\tfrac{1}{r},1+u_r)\approx \Gamma(1+\tfrac{1}{r})\,\tfrac{n+1}{(u_r+1)^{1/r}}\\
\EE(K_r\zwei)&\leq&\tfrac{n}{r} \left(\sum_{i=1}^{r-1} B(\tfrac{i}{r},1+u_r)\right)\\
\EE(R_r)&\leq& w \int_0^\infty e^{-w{s}}(q_r(s))^w\,dt \approx \Gamma(1+\tfrac{1}{r})\,({r!}{w^{r-1}})^{1/r}
\end{eqnarray}
\end{prp}

These upper bounds have the disadvantage that they are not easy to compare to the lower bounds.
The next theorem gives upper bounds that match the lower bounds.

\begin{thm}\label{thmu2} (upper bounds 2)
\begin{eqnarray}
\label{eq4} \EE(K_r\eins)\leq \EE(K_r\zwei)<\tfrac{n}{s_r^{1/r}}\left(\sum_{i=0}^{r-1}{x_1 \choose i}\tfrac{1}{r}\,\Gamma(\tfrac{i+1}{r})\,(s_r)^{-i/r}\right)\\
\label{eq5} \EE(R_r)<\tfrac{n}{\tilde{s}_r^{1/r}}\left(\sum_{i=0}^{r-1}{x_1 \choose i}\tfrac{1}{r}\,\Gamma(\tfrac{i+1}{r})\,(\tilde{s}_r)^{-i/r}\right)
\end{eqnarray}
\end{thm}
\begin{proof}
\ref{eq4} (Sketch) Let  $I=\EE(K_r\zwei)/n$ and let $L(s):=-\sum_{i=1}^m \log(G_r(x_i,s))$.\\$s\mapsto L(s)$ is for $s>0$ strictly increasing.
Substitute $y=(L(s)/s_r)^{1/r}$ then
 $$ I=\int_0^\infty e^{-s_r y^{r}}\,\frac{r s_r y^{r-1}}{L^\prime(s(y))}\,dy$$
We have $L^\prime(s)=r \sum_{i=1}^m {x_i \choose r} (e^s-1)^{r-1}\left(p_r(x_i,e^s-1)\right)^{-1}$. Thus
 $$\tfrac{ L^\prime(s)}{rs_r(e^s-1)^{r-1}}\geq \tfrac{1}{p_r(x_1,e^s-1)}$$
and   $$ I\leq\int_0^\infty e^{-s_r y^{r}}\,\frac{ y^{r-1}}{(e^{s(y)}-1)^{r-1}}\,p_r(x_1,e^{s(y)}-1)\,dy$$
Finally $y< s(y)$ (cmp.\autoref{lema2}  and $s\mapsto p_r(x_1,e^s-1)/(e^s-1)^{r-1}$ is strictly decreasing so that
$$ I\leq\int_0^\infty e^{-s_r y^{r-1}}\,\frac{ y^{r-1}}{(e^{y}-1)^{r-1}}\,p_r(x_1,e^{y}-1)\,dy
<\int_0^\infty e^{-s_r y^{r-1}}\,p_r(x_1,e^{y}-1)\,dy$$
The result now follows by termwise integration.\\
For \ref{eq5} the proof is similar. 
\end{proof}

\begin{exl}
Let us compare the bounds from \autoref{thml} resp. \autoref{thmu2} for the case $r=2$ to some existing bounds in the literature.\\[0.1cm]
(a1) The classical birthday problem.($n=m,x_1=\ldots x_m=1)$. Here we find
$$0< \EE(R_2) -\sqrt{\tfrac{\pi}{2}\, m}< 1$$
Wiener (Theorem 2) gives:
$$-\tfrac{2}{5}< \EE(R_2)-\sqrt{\tfrac{\pi}{2}\, m}<\tfrac{8}{5}$$
while the exact best bounds are known to be (\cite{Sg},\cite{Stong})
$$\tfrac{2}{3}<  \EE(R_2)-\sqrt{\tfrac{\pi}{2}\, m} \leq 2-\sqrt{\tfrac{\pi}{2}} $$
Clearly all of these bounds are comparable. In this case even the complete asymptotic expansion of $\EE(R_2)$ is known
(\cite{Kn}).\\[0.1cm]
(a2) The birthday problem with unequal probabilities.$ (n=m,x_1+\ldots x_m=m)$. Let $p_i:=\tfrac{x_i}{m}$ and let $\beta(p):=\tfrac{1}{\sum_{i=1}^m p_i^2}$, $p_1:=\max\{p_i\}$.
Here we find $$0< \EE(R_2) -\sqrt{\tfrac{\pi}{2}\,\beta(p)}< p_1\,\beta(p)$$
Wiener (Theorem 4) gives:
$$\sqrt{\tfrac{\pi}{2}\,\beta(p)}-\tfrac{2}{5}< \EE(R_2)<2 \sqrt{\beta(p)}$$
Here the lower bounds are comparable while our upper bound is clearly better.\\[0.1cm]
(b) The collision time for a concrete $(n,m)$- function. Here we find
$$\tfrac{n+1}{\sqrt{s_2+1}}\tfrac{\sqrt{\pi}}{2}<\EE(K_2\eins) < \tfrac{n}{\sqrt{s_2}}\tfrac{\sqrt{\pi}}{2}+\tfrac{nx_1}{2s_2}$$
while Wiener (Theorem 8) gives for $n>m\geq 1$, translated into our notation:
$$(e-2)\sqrt{\tfrac{n(n-1)}{2s_2}}< \EE(K_2\eins)\leq 2\ \sqrt{\tfrac{n(n-1)}{2s_2}}$$
Here both our bounds are preferable. 

\end{exl}

Again some remarks are in order:\begin{enumerate}
\item observe that the bounds in \autoref{thml} resp. \autoref{thmu2}  are very close in the sense : upper bound = lower bound + terms of at most the same order.
The expectations of $K_r$ resp. $R_r$ are always of the order $\tfrac{n}{(s_r)^{1/r}}$ resp. $\tfrac{n}{(\tilde{s}_r)^{1/r}}$.
\item the expectation of the $r$-collision time for a fixed configuration with given $s_r$ is comparable to the
expectation of the $r$-collision time of a uniform random mapping with ``effective" image size $m_r=\tfrac{n^r}{s_r}$ ( this was for $r=2$ already remarked by Wiener).
In the sequel we will see that this analogy goes very far. 
\item the expectations of the $r$-collision times are always of the same order of magnitude 
\item in contrast, even the orders of magnitude of $\EE(K_r)$ and $\EE(R_r)$ can be different. The reason is that every $x_i>0$ adds to $\tilde{s}_r$ while only the $x_i\geq r$ add to $s_r$.
While the $r$-repetition time can maximal be of order $(\min\{m,n\})^{(r-1)/r}$, the $r$-collision times can be of order $n$. 
\end{enumerate}

\section{Limit theorems}
What is the typical shape of the distribution of the waiting times for large domains/codomains? This question is answered by limit
theorems.
We consider in the sequel a sequence $(h_N)$ of mappings with corresponding parameters $(n,m,(x_i))=(n(N),m(N),(x_i(N))$ varying with $N$ (it is clear that the distributions only depend on these
parameters), and the corresponding waiting time distributions.

For the formulation we assume w.l.o.g. that 

$$ x_1\geq x_2 \geq x_3 \geq\ldots$$
It will be become apparent that limit theorem can conveniently be formulated using the following
characteristics :

$$s_r:=\sum_{i=1}^m {x_i \choose r}\;,\; m_r=\frac{n^r}{r!\,s_r} \;\;\mbox{ and } \rho_i:=\left(\frac{{x_i \choose r}}{s_r}\right)^{1/r}$$
``Asymptotical"  means:  $m_r(N)\lra \infty$ and/or $s_r(N)\lra \infty$  as  $N\lra \infty$.\\ ( In particular in all asymptotical considerations 
$n(N)\lra \infty$.)\\

\begin{thm}\label{thmasym} {(asymptotical distribution of the collision times)}\\
Let $N\lra \infty$ and assume that for each $i$ the limit $\rho_i=\lim\rho_i(N)$ exists.
Then limiting distributions of the collision times exist and are as follows:
\begin{enumerate}
\item if $s_r(N)\lra \sigma <\infty , m_r(N)\lra \infty $ we have : \\ asymptotically, there are only finitely many $i$ with $x_i\geq r$, for each of these $i$ the 
 limit $\lim x_i(N)=x_i<\infty$ exists
and G
$$\begin{array}{rrcll}(a)&\;\;\PP(K_r\eins>nt)&\lra& \prod_i G_r(x_i,t)\;\;&\mbox{ for }0\leq t < 1\\[0.6cm]
(b)&\;\;\PP(K_r\zwei>nt)&\lra& \prod_i G_r(x_i,1-e^{-t})\;\;&\mbox{ for }0\leq t < \infty\end{array}$$\\[0.1cm]
\item  if  $s_r(N)\lra \infty , m_r(N)\lra \infty $ we have (in both cases): 
$$\;\;\PP(K_r> (m_r)^{(r-1)/r}\; t)\lra e^{-(1-\sum_i \rho_i^r)t^r/r!} \prod_i e^{-\rho_it} q_r(\rho_i t)\;\;\mbox{ for }0\leq t < \infty$$\\[0.1cm]
\item  if $s_r(N)\lra \infty , m_r(N)\lra \lambda<\infty $ we have (in both cases) :\\  for each $i$ the limit
$\lim \frac{x_i(N)}{n(N)}=:p_i$ exists , $\sum_i p_i=1$ and for each  $k\in\mathbb{N}$:
$$ \PP(K_r>k)\lra k!\,[t^k] \prod_i q_r(p_it) e^{-p_it}$$
\end{enumerate}

\end{thm}
\begin{proof} 1.(a) suppose $\lim \rho_i(N)$ exists for each  $i$ and that $s_r(N)\lra \sigma<\infty$.
Let  $h_r(N):=\max\{i\;: x_i(N)\geq r\}$ the number of higher occupied cells, let $a_r(N):=\sum_{i=1}^{h_r(N)} x_i(N)$ the number of balls in higher occupied cells and 
let $$H_N(t):=\prod_{i=1}^{m(N)} p_r(x_i(N),t) (1+t)^{-x_i(N)}$$
 Since $\sigma<\infty$ asymptotically
 there are  only finitely many $i$ with $x_i\geq r$. Thus $h_r=\lim h_r(N)$ exists and is finite, and for each $i\leq h_r$ the limit $x_i(N)$ exists and is finite.
 Thus there
is an $N_0$ s.th. $x_i(N)$ is constant for $i\leq h_r$ and $N\geq N_0$. Hence $a_r(N)=:a_r$ is also constant for $N\geq N_0$. For $N\geq N_0$ then
$Q_N(t):=H_N(\frac{t}{1-t})$ is a polynomial of degree  $a_r$ in $t$ and we have 
$$ \PP(K\eins>k)=\xnk Q_N(\frac{x}{1+x})\left(1+x\right)^n=\sum_{i=0}^{a_r} q_i { k \choose i} \frac{i!\,  (n-i)!}{n!}$$
where the coefficients $q_i$ of $Q_N$ depend only on the $x_i$ with $x_i \geq r$ and hence are constant for $N>N_0$. For $\frac{k}{n}\lra t\in (0,1)$ the quotients of the
binomial coefficients converge to corresponding powers of $t$. Thus
$$ \PP(K_r\eins>nt)\lra Q(t)=H(\frac{t}{1-t})\;\;\;(n\lra \infty)$$
The proof for (b) is analogous.\\
(2)(3) Let $Q_N(t):=\prod_{i=1}^{m(N)} G_r(x_i(N),t)$ (so that $1-Q_N(t)$ is the distribution function of $\min\{X_1,\ldots,X_m\}$ where
the $X_i$ are independent, and $X_i$ is distributed as the $r-$th order statistic of $x_i$ independent variables uniform on $[0,1]$). 
and let  $T_r\eins(N)$ resp. $T_r\zwei(N)$ be random variables with $\PP(T_r\eins(N) >t)=Q_N(\frac{t}{1-t})$ resp.  $\PP(T_r\zwei(N)>t)=Q_N(1-e^{-t})$.
 We show that the limits of the distributions of  $T_r\eins(N)$ resp. $T_r\zwei(N)$ appear as limits of the distributions of $K_r\eins(N)$ resp. $K_r\zwei(N)$
(thus reducing to the case of the asymptotic distribution of the minimum of independent random variables)\\

We first show  two auxiliary results:
\begin{prp}
Let  $T_N$ be a sequence of integer-valued non-negative random variables with corresponding generating functions $g_N(u):=\sum_{i=0}^\infty u^i \PP(T_N>i)$ and $c_N$ a sequence of positive numbers with $c_N\lra 0$. Then ${c_N T_N}$ converges in distribution to a random variable $T$ with $\PP(T>t)=G(t)$ iff for each $p>0$:
$$c_N\,g_N(e^{-{p}{c_N}})\lra \int_0^\infty e^{-py}\, G(y)\,dy$$
\end{prp}
\begin{proof}
Let $f_N$ denote the probability generating function of $T_N$, so that $g_N(s)=\frac{1-f_N(s)}{1-s}$, and let $\ell_T(p):=\EE(e^{-pT})$ denote the Laplace transform of $T$.\\
By the continuity theorem for Laplace transforms $T_N\lra T$ in distribution iff $\ell_{T_N}(p)\lra \ell_T(p)$ for each $p>0$.
"$\Rightarrow$": let $c_NT_N\lra T$ in distribution. Then $\ell_{c_N\,T_N}(p)=f_N(e^{-p c_N})\lra \ell_T(p)$ for each $p\geq 0$. Hence for $p>0$
$$ c_N\, g(e^{-p c_N})=\frac{(1-f_N(e^{-p c_N}))\,c_N}{(1-e^{-p c_N})}\lra \frac{1-\ell_T(p)}{p}=\int_{0}^\infty e^{-py} G(y)\,dy$$
"$\Leftarrow$": let for each $p>0$
$$c_N g(e^{-p c_N})=\frac{(1-f_N(e^{-pc_N}))\,c_N}{(1-e^{-p c_N})}\lra \frac{1-\ell_T(p)}{p}$$
Then for each $p>0$
$$1-f(e^{-p c_N})=c_N\,{g(e^{-p c_N})}\frac{(1-e^{-p c_N})}{c_N}\lra \frac{1-\ell_T(p)}{p}\,p=1-\ell_T(p)$$

\end{proof}
We have for $|u|<1$:

$$g_{r,N}\eins(u)=:\sum_{k=0}^\infty u^k \PP(K_r\eins(N)>k)=(n+1)\,\int_0^{\infty} \left(1+(u-1)\frac{y}{1+y}\right)^{-(n+2)}\,Q_N(uy)\,dy$$
and 
$$g_{r,N}\zwei(u):=\sum_{k=0}^\infty u^k \PP(K_r\zwei(N)>k)=n\,\int_0^\infty e^{-n({1-u})y}\, Q_N(1-e^{-uy})\,dy$$

Combining this with the foregoing proposition gives:
\begin{prp}\label{propas}
Let $s_N$ be a sequence of positive numbers s.th $s_N\lra \infty$,\;$\frac{s(N)}{n(N)}\lra 0$, and let $T$ be a nonnegative random variable. Then we have:\begin{enumerate}
\item if $s_N\,T_r\eins(N)$ converges in distribution to $T$ then $\frac{s_N}{n(N)+1}\,K_r\eins(N)$ converges in distribution to $T$.
\item if $s_N\,T_r\zwei(N)$ converges in distribution to $T$ then  $\frac{s(N)}{n(N)}\,K_r\zwei(N)$ converges in distribution to $T$.
\end{enumerate}
\end{prp}
\begin{proof} Let $0\leq u<1$ and for $t\geq 0$ let $G(t):=\PP(T>t)$. 

Let $c_N:=\frac{s_N}{n+1}$ and let $u:=e^{-p c_N}$. Then
$$c_N g_{r,N}\eins(e^{-p c_N})=\int_0^\infty \left(1+(e^{-c_Np}-1)\frac{y}{s_N+y}\right)^{n+2}
 \,Q_N(e^{-c_Np} y/s_N)\,dy$$
Since $s_NT_r\eins(N)\lra T$ in distribution and since $c_N\lra 0$ also $e^{c_Np}s_N\,T_r\eins(N)\lra T$. 
Thus $Q_N(e^{-c_Np} y/s_N)\lra G(y)$ at all continuity points $y$ of $G$. Further $$(1+(e^{-p c_N}-1)\frac{y}{s_N+y})^n\lra e^{-py}\;\;.$$ Moreover, the convergence is dominated since
 each $Q_N\leq 1$ and
$$(n+2)\log(1-(1-e^{-c_Np})\frac{y}{y+s_N})\leq -(n+2)(1-e^{-c_Np})\frac{y}{s_N+y}\leq -(n+2)\,c_Np\,e^{-c_Np}\frac{y}{s_N+y}$$
and therefore we can find a constant $c>0$ s.th.for $y\geq 1$
$$(n+2)\log(1-(1-e^{-c_Np})\frac{y}{y+s_N})\leq -cyp$$
The dominated convergence theorem and proposition 6.2. now give the assertion.\\
The proof for (2) is similar.

\end{proof}

Now let us prove (2) and (3):\\
assume that for each $i$ the limit $\lim \rho_i(N)=\rho_i$ exists, and that $s_r(N)\lra \infty$. It is then a routine matter 
to show that for $s_N:=s_r(N)$ the random variables $s_N T_r(N)$ converge to the given distributions, and the additional assertions for (3) are also easy to show. 
An application of \autoref{propas} finishes the proof.

\end{proof}

Again some remarks are in order:
\begin{enumerate}
\item thus there are essentially three different types of limiting distributions and one will expect that for large domains and codomains the description by one of the types applies. Roughly speaking type (1) describes ``almost injective" functions
($n<<m$), type (2) describes ``normal" functions ($n\approx m$) and type (3) describes ``very surjective" functions

\item in any case $n/(s_r)^{1/r}$ is the correct measure  for the appearance of the first collisions.
The form of the distribution is easy to understand: the preimages which are small compared to this measure add to the ``Weibull" factor, while the ``large" preimages give the other factors. For ``normal" mappings all preimages will be small compared
to $n/(s_r)^{1/r}$, and the limiting distribution the Weibull-$r$ distribution.

\item  for the limiting distributions of $R_r$ a completely analogous assertion as in part (2) holds (as may be deduced from theorems given by Camarri \cite{Ca} (case $r>2$),
Camarri\& Pitman\cite{CP} (case $r=2$)):\\
if $v_r:=(\sum_{i=1}^m x_i^r)\lra \infty$, $\tilde{m}_r:=n^r/v_r\lra \infty$  and for each $i$ 
the limit $x_i/(\tilde{m}_r)^{1/r}=\theta_i$ exists, then
$$\;\;\PP(R_r> (\tilde{m}_r)^{(r-1)/r}\; t)\lra e^{-(1-\sum_i \theta_i)t^r/r!} \prod_i e^{-\theta_it} q_r(\theta_i t)\;\;\mbox{ for }0\leq t < \infty$$
Of course, both limits (i.e. $\theta_i$ and $\rho_i$) can exist at the same time and need not necessarily be the same.
 The special case of this theorem for uniformly small cells (i.e. $\theta_1=0$) is (at least implicitly) known for a long time.(See \cite{Ko}, Theorem 1 in III,\S 3 )

\item
It can be shown that in case (3) $R_r$ has the same limit as $K_r$.
\item
since $\tilde{m}_r\leq\min\{m^{r-1},n^{r-1}\}$ an analogon to Thm 6.1,(1) cannot exist for $R_r$.
\item in cases (2) and (3) there is thus asymptotically no difference between drawing with or without replacement.
In case (3) there is also never a  difference between repetitions and collisions.  
\end{enumerate}

\section{Supplementary Considerations}
\subsection{Asymptotic of $n/(s_r)^{1/r}$} 
It is clear from the above that $n/(s_r)^{1/r}$ characterises the behaviour of the $r$-collision times. How is the
value of $s_r$ distributed over all $(n,m)$-mappings? We consider again a uniform multinomial configuration
$(X_1,\ldots,X_m)$ with parameters $n$ and $p_1=\ldots=p_m=\frac{1}{m}$ and let $$S^{(n,m)}_r=\sum_{i=1}^m{X_i \choose r}$$
the random variable ``no. of $r$-collisions".
\begin{prp}\label{prpcon}
If $n,m \lra \infty$ s.th. $\frac{m^{r-1}}{n^r}\lra 0$ then
$$\frac{m^{r-1}}{n^r} S^{n,m}_r\lra \frac{1}{r!}\;\; \mbox{in distribution}$$
\end{prp}
\begin{proof}
 Calculations show that
\begin{equation} \label{expcoll}
\EE(S^{(n,m)}_r)=\frac{1}{m^{r-1}}{n\choose r} \end{equation} 
$$\var((S^{(n,m)}_r)=\left(\sum_{i=0}^{r-2}{r \choose i}{n-r \choose i}\frac{1}{m^{i}}+\left[{n-r+1 \choose r}
 -{n \choose r}\right]\frac{1}{m^{r-1}}\right) \EE(S^{(n,m}_r)$$
It is now easy to see that $\EE(\frac{m^{r-1}}{n^r}S^{(n,m)}_r)\lra \frac{1}{r!}$ and
$\var( (\frac{m^{r-1}}{n^r}S^{(n,m)}_r)\lra 0$

\end{proof}
Thus if $n$ is large compared to $m^{(r-1)/r}$ the effort for $r$-collision search will  
practically for every $(n,m)$ function be of order $m^{(r-1)/r}$.
\subsubsection{Balance measures}
For uniform random mappings each colour $i$ will appear with the same probability. The idea to measure the distance 
of a configuration from uniformity by a ``balance measure" is near at hand. The classical statistic in this respect
is the test-statistic of the $\chi^2$ test, which in the case of a $(n,m)$ uniform random mapping is:
$$T:=\frac{m}{n}\sum_{i=1}^{m} (X_i-\frac{n}{m})^2=\left(\frac{m}{n}\sum_{i=1}^m X_i^2\right)-n$$
We know from the above that $n^2/S^{(n,m)}_2$ (or variants thereof) measures the performance of $2$-collisions 
attacks.
We have
$$T=m-n+\frac{2m}{n}S_2^{(n,m)}$$
Bellare and Kohno suggest to use
$$\mu_2:= -\log_m(\sum_{i=1}^m X_i^2/n^2)$$
to quantify resistance against a generic collision attack.
This balance measure is related to $2$-repetitions
rather then to $2$-collisions (this was earlier remarked by Wiener \cite{Wi} and by Ramanna and Sarkar \ci{RS}).
Clearly $\mu_2$ is a simple variant of the $\chi^2$ statistic $T$:
$$\mu_2=1-\log_m(1+T/n)$$
It is well known that $\EE(T)=m-1$, $\var(T)=(1-1/n)(2m-2)$. Thus if $m,n\lra \infty$ s.th. $m/n^2\lra 0$ then
$Var(T/n)\lra 0$ and $\mu_2\approx 1-\log_m(1+\frac{m-1}{n})$. E.g for  $n=a m$ and large $m$ this measure will
have the value $\mu_2=1-\log(1+1/a)/\log(m)$ for practically every $(n,m)$-function.\\ 
Ramanna and Sarkar suggest to use
$$\Lambda_r:=-\frac{1}{r-1}\log_m(\frac{r!S_r^{(n,m)}}{n^r})$$
to quantify resistance against a generic $r-$collision attack.\\
By \autoref{prpcon}, if $m,n\lra \infty$ s.th. $m^{r-1}/n^r\lra 0$ then $\var(r!m^{r-1} S_r^{(n,m)}/n^r)\lra 0$. Thus if $m^{r-1}/n^r$ is small  
this measure will have the value 
$\Lambda_r\approx 1+\log_m(\frac{n^r}{(n)_r})$ for practically every $(n,m)$ function.\\
From a probabilistic view these measures fail to uncover irregularities in a uniform random function: their scaling is too coarse. They only detect
extreme deviations from uniform random behaviour.

\subsection{On the difference between drawing with/without replacement}
\subsubsection{Probability for true collisions}
In this subsection we only consider drawing with replacement and write $K_r$ for $K_r\zwei$.
Here it is of interest to know the probability $\PP(K_r=\min\{K_r,R_r\})$ that the first
$r$-hit is caused by a true multi-collision. Let $E_{i,r}$ the event: $i$ is the first colour which is drawn
$r$-times (i.e. $E_{i,r}=\{X_{\min\{K_r,R_r\}}=i\}$). It is easy to show that
$$\PP(K_r=R_r\,|\,E_{i,r})=\frac{(x_i)_r}{x_i^r}$$
If $\ell:=\min\{x_i\,|\, x_i\geq r\}$ and $M:=\max\{x_i\}$ we therefore have
$$\frac{(\ell)_r}{\ell^r}\leq \PP(K_r=R_r)\leq \frac{(M)_r}{M^r}$$
For $r=2$ we can be more precise:
\begin{thm}
Let $b:=|\{i\,:\,x_i>0\}|$ the number of occupied images. Then the following inequalities hold
$$ \frac{n}{\sum_{i=1}^m x_i^2}\leq \PP(R_2<K_2)\leq \frac{b}{n}$$
Equality (on both sides) holds iff all positive $x_i$ are equal.
\end{thm}
\begin{proof}(Sketch)
Using generating functions one can derive that
$$\PP(K_r=R_r)=\sum_{i=1}^m r{x_i\choose r} \int_0^\infty t^{r-1} e^{-nt}\,\prod_{j\neq i} q_r(x_jt)\,dt$$
Using the fact that $\sum x_i^2 t e^{-nt} \prod_{j\neq i} q_2(x_it)=-(\prod_{i=1}^m F_2(x_it))^\prime$
where $F_2(t)=q_2(t)e^{-t}$ then
$$\PP(R_2<K_2)=\sum_{i=1}^m x_i \int_0^\infty t\,e^{-nt} \prod_{i\neq j} q_2(x_it)\,dt$$

Let $p_i:=\frac{x_i}{n}$ and $q:=\sum_{i=1}^m p_i^2$ and rewrite the equality above as
$$\begin{array}{rcl}
n\,\PP(R_2<K_2)&=&\int_0^\infty \left(\sum_{i=1}^l \frac{p_i s}{1+p_is}\right) e^{-s}\,\prod_{i=1}^l (1+p_i\,s)\,ds\\
      &=
&\int_0^\infty \left(\sum_{i=1}^l \frac{p_i s}{1+p_is}\right) \exp(-s+\int_0^s \sum_{i=1}^l \frac{p_i}{1+p_it}\,dt)\,\,ds\\
            \end{array}$$
(1) By Jensen's inequality $\sum_{i=1}^l \frac{p_i}{1+p_i\,s}\geq \frac{1}{1+qs}$, and therefore:
$$\begin{array}{rcl}
n\,\PP(R_2<K_2)
      &=
&\int_0^\infty \left(\sum_{i=1}^m \frac{p_i s}{1+p_is}\right) \exp(-s+\int_0^s \frac{p_i}{1+p_it}\,dt)\,\,ds\\
&\geq& \int_0^\infty \frac{s}{1+qs} \exp(-s+\int_0^s \frac{1}{1+qt}\,dt)\,\,ds\\
&=& \int_0^\infty \frac{s}{1+qs} e^{-s} \left(1+qs\right)^{\frac{1}{q}}\,ds\\
&=& \frac{1}{q}
            \end{array}$$
(2) Another application of Jensen's inequality gives: $\sum_{i=1}^m \frac{p_is}{1+p_is}\leq \frac{s}{1+\frac{s}{b}}$. A similar
computation then shows 
$$\begin{array}{rcl}
n\,\PP(R_2<K_2)
      &\leq& b\;\;\;
            \end{array}$$ \end{proof} 
\subsubsection{Difference of expectations}

Next we look at the order of magnitude of the difference of the collision times. We already saw that it is always
preferable to use sampling without replacement, and that there is asymptotically no difference (on the $n/s_r^{1/r}$ scale)
when both $n$ and $s_r$ tend to infinity. Here we quantify the difference more precisely.

\begin{thm}
There is a constant $C_r$ (depending only on r) s.th.
$$\EE(K_r\zwei)-\EE(K_r\eins)< C_r n/(s_r)^{2/r}$$ 
\end{thm}
\begin{proof} (Sketch) We have
$$0<\EE(K_r\zwei)-\EE(K_r\eins)<n\,\int_0^\infty (1-e^{-t}) \prod_{i=1}^{m} G_r(x_i,1-e^{-t})\,dt =:n\,I$$
By lemma 9.4 from the appendix $\prod_{i=1}^{m} G_r(x_i,1-e^{-t})\leq \left(e^{-td/s_r})q_r(dt/s_r)\right)^{rs_r^{r+1}/d^r}$
where $d=\sum_{i=1}^m x_i{x_i \choose r}$ leading after simple steps to
$$I\leq \frac{s_r^2}{d^2} \int_0^\infty t \left(e^{-t}q_r(t)\right)^{rs_r^{r+1}/{d^r}}\,dt$$
Using $rs_r^{r+1}/d^r\geq 1/r^{r-1}$ and observing that $c\mapsto c^{2/r}\int_{0}^\infty t \left(q_r(t) e^{-t}\right)^c\,dt$ is
decreasing finally shows that the assertion is true for
$$C_r=(1/r)^{2/r}\int_0^\infty t \left(e^{-t} q_r(t)\right)^{1/r^{r-1}}\,dt$$
\end{proof}
Thus the difference of expectations is $O(\frac{n}{(s_r)^{2/r}})$ (it is not hard to show that it is even $\Theta(\frac{n}{(s_r)^{2/r}})$ in Knuth's sense).
This gives an independent proof for the fact that $\frac{s_r^{1/r}}{n}(K_r\zwei-K_r\eins)\lra 0$ if $s_r\lra \infty$.

\subsection{Cyclic points/$\rho-$length for random fixed-indegree mappings}$ $\\
The distribution of $K\eins_2$ is closely related to two graph-theoretic distributions of random mappings.
Let ${\mathcal D}=\{ 1,\ldots,n \}$ and $f:\mathcal{D}\lra \mathcal{D}$ an arbitrary function.\\ The functional digraph $G_f$ of $f$ is the graph with 
vertices $\{1,\ldots,n\}$ 
and directed edges 
$\{(1,f(1),\ldots, (n,f(n))\}$.\\
 $G_f$ consists of a number of connected components, each of which contains a unique cycle. In this way graph-theoretic terms can be applied to $f$,
e.g. the cyclic points of $f$ are the cyclic points in $G_f$ etc.
The indegree $x_i$ of vertex $i$ in $G_f$ is 
 $x_i=|f^{-1}(\{i\})|$.\\
Let in the sequel $c:=(x_1,\ldots,x_n)$ be a fixed indegree sequence (``configuration"), and let $${\mathcal F}_c:=\{ f:{\mathcal D} \lra {\mathcal D}\;:\;|f^{-1}(\{i\})|=x_i \mbox{ for all }\;i\,\}$$ the set of mappings which share the same configuration. Let $Z(f)$ be the r.v. ``number of cyclic points of $f$" on ${\mathcal F}_c$, let $\rho(x,f)$ the r.v. on ${\mathcal D}\times {\mathcal F}_c$  ``$\rho$-length of $x$ under $f$,
let $K$ be the r.v. ``waiting time for the first collision" when drawing without replacement from an urn with configuration $c$ (i.e. $K=K\eins_2$, which was considered above).  
The following theorem holds:

\begin{thm}{Let $k\in\{0,\ldots,n\} $}.\\
(a) For the uniform distribution on ${\mathcal F}_c$ 
 $$\PP(Z> k) = \PP(K > k+1)$$
(b) For the uniform distribution on ${\mathcal D}\times {\mathcal F}_c$ $$\PP(\rho>k) = \frac{n-k}{n}\, \PP(K > k)$$
\end{thm}
\begin{proof} We use  variants of Pr\"ufer-coding.\\
(a)(Foata-Fuchs) Each mapping from $\mathcal{F}_c$ with leaves $y_1<\ldots< y_k$ can uniquely be encoded as a list of words (where word :=list without repeated element) $w_0\ldots w_{k}$,
where $w_0$ consists of the cycles of $f$ (each coded as a word, starting with the largest element, and listing the remaining elements ``against"
the mapping direction, and these subwords concatenated in the ordering of their first elements), and the next word $w_i$ is the ``path" from $l_i$ to its root
$x_i$ in 
$w_0\ldots,w_{i-1}$, without $y_i$, but including it's root $x_i$ in $w_0\ldots w_{i-1}$, and listed against the mapping arrow, starting from $x_i$.\\ 
This coding gives a bijection between the set of mappings from ${\mathcal F}_c$ with $r$ cyclic points and the set of sequences with configuration $c$
and first repeated value at $r+1$.\\ 
(b)  Each of the $n {n \choose x_1,\ldots x_n} $ possible choices of $x,\,f$ ($x$ starting point, $f\in \mathcal{F}_c$  mapping) can be encoded as a
sequence of length $(n+1)$: $$\left(x=f^0(x),f(x),\ldots, f^r(x),f(y_1),\ldots ,f(y_{N-r})\right)$$ where $r$ is the smallest iterate where
a previous element of the list is repeated, and $y_1<\ldots<y_{N-r}$ are the remaining elements (clearly $x$ and $f$ can be reconstructed from this list).\\
Thus there is a bijection
between\\
(1) sequences of length $n+1$, where the first element (no. 0) is drawn from $\{1,\ldots,n\}$, and the remaining elements are drawn without replacement from an urn with
configuration $c$, with first repeated value at ``time'' $r+1$ , and\\
(2) pairs $(x,f)$ where $x$ has rho-length $r$ under $f$.\\
Under this bijection, $\rho(x,f)>k$ iff $x_0\not\in\{ x_1,\ldots,x_k\}$ and $x_1,\ldots,x_k$ are pairwise distinct.

\end{proof}
In an equivalent form (a) and (b) were earlier shown by Hansen and Jaworski (\ci{HJ}). The asymptotic behaviour
of $Z$ resp. $\rho$ can now directly be read off from \autoref{thmasym}, which generalises the corresponding results for $Z$ and $\rho$ in \ci{AB}.

\subsection{Application to generic collision search in hash functions} $  $\\
``Real" hash functions $H$ have (in principle) an infinite domain, and a finite range  $\mathcal{R}$ of cardinality $|{\mathcal R}|=m=2^{\ell}$. However, restricted to an arbitrary finite input set they are concrete mappings with finite domain and range, and 
the theory given here applies.\\
If the attacker has no a priori knowledge about $H$ and is given only ``black box" access to $H$,
 the best
he can do is to try to find collisions 
using ``drawing without replacement", i.e. hashing randomly chosen (non-repeating) input strings.
In practice he has to restrict the possible input strings to a finite set $\mathcal{D}$ of cardinality $n$ (e.g. all input strings up to a certain maximal bitlength) and 
$h:=H|\mathcal{D}$ is a fixed $(n,m)$ function. How many strings does he have to hash to find a collision? We know
that $n/\sqrt{s_2}$ measures the effort for collision search.
If $n$ is large compared to $\sqrt{m}$ the effort will be of order $\sqrt{m}$ for practically every $(n,m)$ function (see \ref{prpcon}  and theorem \ref{thmasym},(2)). If $n$ is of order $\sqrt{m}$ in general at most a few collisions will exist.
 If they exist the trial of order $\sqrt{m}$ is needed to find one (thm \ref{thmasym}, (1)). Finally, if $n$ is small compared to $\sqrt{m}$ collisions 
are unlikely to exist (see the expected no. of collisions \eqref{expcoll}).\\ 
Thus - unless the design of $H$ is fundamentally flawed in the sense that the order of magnitude of $s_2(h)$ is larger than for a typical random $(n,m)$ function - 
the typical effort will be of order $\sqrt{m}$.
This is the basis for the  folklore belief, that generic collision search (for a well designed hash function with codomain size $m$) needs an effort of $\sqrt{m}$. (In the same vein, a 
generic $r-$collision search (small, fixed $r$) needs an effort of $m^{(r-1)/r}$).\\
The plausibility of the $\sqrt{m}$-effort here rests on two assumptions:
\begin{enumerate} 
\item the design of $H$ ensures that the order of magnitude of $s_2(H|\mathcal{D})$ is - for ``canonical"  (that is: easily specifiable, and not using specific properties of $H$)" preimage sets $\mathcal{D}$ of size $n$ - comparable to that of the $s_2$ of a random $(n,m)$ function
\item the attacker lacks the ability to specify a ``favourable" preimage set
\end{enumerate}
The first condition requires that $H$ is well designed from a statistical 
point of view, and the second condition requires that $H$ resists cryptanalysis.\\
The extent to which these conditions are fulfilled for a concrete hash function is debatable. 
If the attacker has a priori knowledge about $H$ he may of course find specific attacks. Especially he may  
in this case be able to find a set $\mathcal{A}$ of input strings s.th. a statistical collision attack on  $h:=H|{\mathcal A}$ is ``easy". 
One of the main aims of hash function design is to make it ``practically infeasible" for an attacker to determine such input sets. (Although is theoretically clear that such ``favourable" input sets
exist.)

\providecommand{\bysame}{\leavevmode\hbox to3em{\hrulefill}\thinspace}
\providecommand{\MR}{\relax\ifhmode\unskip\space\fi MR }
\providecommand{\MRhref}[2]{%
  \href{http://www.ams.org/mathscinet-getitem?mr=#1}{#2}
}
\providecommand{\href}[2]{#2}

\section{Appendix}

In this appendix we collect some inequalities.

\begin{lem}\label{lema1}
Let $X$ be binomial distributed with parameters $n$ and $p$.
For $r<n$
$$\PP(X<r)\geq (1-p^r)^{n \choose r}$$
\end{lem}
\begin{proof} 
Let $f(p):=\log(\PP(X<r)) -{n \choose r}\log(1-p^r)$ the log of the quotient lhs/rhs. We find
$$f^\prime(p)=-r\,{n \choose r} \left(\frac{p^{r-1} (1-p)^{n-r}}{\PP(X<r)}-\frac{p^{r-1}}{1-p^r}\right)
 =r\,{n \choose r}p^{r-1} \left(\frac{\PP(X<r)-(1-p)^{n-r}(1-p^r)}{\PP(X<r)(1-p^r)}\right)$$
The numerator of the rhs is of the form $\PP(S_n\leq r-1)-\PP(S_{n-r}=0)\PP(S_r<r-1)$ and is therefore
nonnegative. Thus $f(p)\geq f(0):=0$ , i.e. the quotient is $\geq 1$.
\end{proof}

\begin{lem}\label{lema2}
For $s\geq 0$, $n>r$ 
$$-\log G_r(n,1-e^{-s})\leq {n \choose r} s^r$$
\end{lem}
\begin{proof}
Let $f(s):=-\log(G_r(n,1-e^{-s})) -{n \choose r} s^r$. We find
$$f^\prime(s)=r\,{n \choose r}\left(\frac{(e^s-1)^{r-1}}{p_r(n,e^s-1)} -s^{r-1}\right)$$
Clearly $p_r(n,e^s-1)\geq p_r(r-1,e^s-1)=e^{(r-1)s}$. Thus 
$$f^\prime(s)\leq (1-e^{-s})^{r-1} -s^{r-1}\leq 0$$
Thus $f(s)\leq f(0):=0$.
\end{proof}

\begin{lem}\label{lema3}
Let $F_r(s):= q_r(s)e^{-s}$.
Then for $s>0$
$$-\log(F_r(s))\leq \frac{s^r}{r!}$$
\end{lem}
\begin{proof} Similar as above. \end{proof}

\begin{lem}\label{lema4}
Let $d:=\sum_{i=1}^m x_i{x_i \choose r}$ and $s_r:=\sum_{i=1}^m {x_i \choose r}$.
Then for $t>0$ 
$$ \prod_{i=1}^m G_r(x_i,1-e^{-t})\leq \left(e^{-t}q_r(\frac{d}{s_r}t)\right)^{rs_r^{r+1}/d^r}$$ 
\end{lem}
\begin{proof}
Let  
$ f(t):=\sum_{i=1}^m \log(G_r(x_m,1-e^{-t})$. We find
$$f^\prime(t)=-\sum_{i=1}^m r {x_i \choose r}\frac{ (e^t-1)^{r-1}}{p_r(x_i,e^t-1)}\leq -\sum_{i=1}^m r {x_i \choose r}\frac{t^{r-1}}{q_r(x_it)} $$
Since $t\mapsto 1/q_r(t)$ is convex on $\mathbb{R}_+$ Jensen's inequality gives
$\sum_{i=1}^m {x_i \choose r} \frac{1}{q_r(x_it)}\geq \frac{s_r}{q_r(\frac{d}{s_r}t)}$ and so
$$f^\prime(t)\leq -\frac{r s_r t^{r-1}}{q_r(\frac{d}{s_r}t)}= \frac{r s_r^{r+1}}{d^r}(\log F_r(\frac{d}{s_r}t))^\prime$$
where $F_r$ is as in lemma 8.3. Thus the log of the quotient lhs/rhs is decreasing in $t$. 
\end{proof}

\begin{thebibliography}{10}
\bibitem{AB}
{Arney, J. and Bender, E.A.}, \emph{ Random mappings with constraints on
 coalescence and the number of origins}, Pacific J. Math. {103} (1982), pp. ~269-294.

\bibitem{BK}
{Bellare, M. and Kohno, T.}, \emph{Hash Function Balance and its Impact on
  Birthday attacks}. Advances in Cryptology - EUROCRYPT 2004, Lecture Notes in
  Computer Science 3027, pp.~401--419. Springer, Berlin, 2004, Full version
  available at \url{http://eprint.iacr.org/2003/065/}.

\bibitem{Ca}
{Camarri, M.}, \emph{Asymptotics for k-fold repeats in the birthday problem
  with unequal probabilities}, Tech. report 524, Dept. Statistics, U. C.
  Berkley, Report, 1998.

\bibitem{CP}
{Camarri, M. and Pitman,J.}, \emph{{Limit Distributions and Random Trees
  derived from the Birthday Problem with unequal Probabilities}}, The
  Electronic Journal of Probability (2000).

\bibitem{Co}
{Coppersmith, D.}, \emph{Another Birthday Attack}. Advances in Cryptology -
  CRYPTO '85, Lecture Notes in Computer Science 218, pp.~14--17. Springer,
  Berlin, 1985.

\bibitem{deB}
{de Bruijn, N.G.}, \emph{Asymptotic Methods in Analysis (3rd ed.)}. North
  Holland Publishing Co., Amsterdam, 1970, reprinted by Dover, 1981.

\bibitem{DM}
{Diaconis, P. and Mosteller, F.}, \emph{Methods for studying Coincidences},
  Journal of the American Statistical Association {84} (1989), pp.~853--861.

\bibitem{GTF}
{Flajolet, P. and Gardy, D. and Thimonier, L.}, \emph{Birthday Paradox, Coupon
  Collectors, Caching Algorithms and Self-organizing Search}, Discrete Applied
  Mathematics {39} (1992), pp.~207--229.

\bibitem{FlSedg}
{Flajolet, P. and Sedgewick, R.}, \emph{Analytic Combinatorics}. Cambridge
  University Press, Cambridge (UK), 2009.

\bibitem{Gir}
{Girault,A. and Cohen, R. and Campana, M.}, \emph{A Generalized Birthday
  Attack}. Advances in Cryptology - EUROCRYPT '88, Lecture Notes in Computer
  Science 330, pp.~129--157. Springer, Berlin, 1988.

\bibitem{Go}
{Good, I.J.}, \emph{Saddle-point Methods for the Multinomial Distribution}
  , {Ann. Math. Stat.} {28} (1957), pp.~861--881.

\bibitem{HJ}
{Hansen, J.C. and Jaworski, J.}, \emph{Random mappings with exchangeable in-degrees}, Random Struct. Algorithms {33} (2008), 
pp.~105--126.

\bibitem{Ho}
{Holst, L.}, \emph{{On birthday, collectors', occupancy and other classical urn
  problems}}, {International Statistical Review} {54} (1986), pp.~15--27.

\bibitem{Ho2}
{Holst, L.}, \emph{{The General Birthday Problem}}
 , {Random Struct. Algorithms} {6} (1995), pp.~201--208.

\bibitem{JK}
{Johnson, N.L. and Kotz,S.}, \emph{Application of Urn Models}. John Wiley and
  Sons, Inc., Chichester, 1977.

\bibitem{Jo}
{Joux, A.}, \emph{Multicollisions in Iterated Hash Functions. Applications to
  Cascaded Constructions}. Advances in Cryptology - CRYPTO 2004, Lecture Notes
  in Computer Science 3152, pp.~306--316. Springer, Berlin, 2004.

\bibitem{KM}
{Klamkin, M.S. and Newman, D.J.}, \emph{Extensions of the Birthday Surprise},
  Journal of Combinatorial Theory {3} (1967), pp.~279--282.

\bibitem{Kn}
{Knuth, D.E.}, \emph{The Art of Computer Programming, vol. 1 (3rd. ed.)}.
  Addison-Wesley Publishing Company, Reading, Massachusetts, 1997.

\bibitem{Ko}
{Kolchin,V.F. and Sevast'yanov,B.A. and Chistyakov, V.P.}, \emph{{Random
  Allocations}}. V.H. WINSTON \& SONS, Washington, D.C., 1978.

\bibitem{LS}
{Laccetti, G. and Schmid, G.}, \emph{On a Probabilistic Approach to the
  Security Analysis of Cryptographic Hash Functions}, Cryptology ePrint
  Archive, Report no.~324, 2004. Available at
  \url{http://eprint.iacr.org/2004/324}.

\bibitem{SN}
{Nandi, M. and Stinson, D.R.}, \emph{Multicollision Attacks on Some Generalized
  Sequential Hash Functions}, Cryptology ePrint Archive, Report no.~055, 2006.
  Available at \url{http://eprint.iacr.org/2006/055}.

\bibitem{Pr}
{Preneel, B.}, \emph{Analysis and Design of Cryptographic Hash Functions},
  Ph.D. thesis, K.U. Leuven, Leuven, Belgium, 1993.

\bibitem{RS}
{Ramanna, S.C. and Sarkar,P.}, \emph{On Quantifying the Resistance of Concrete
  Hash Functions to Generic Multi-Collision Attacks}, Cryptology ePrint
  Archive, Report no.~525, 2009. Available at
  \url{http://eprint.iacr.org/2009/525}.

\bibitem{Sg}
{Schulte-Geers, E.}, \emph{Problem 11353}, {American Mathematical Monthly}
  {115} (2008), p.~263.

\bibitem{Stin}
{Stinson, D.R.}, \emph{Some Observations on the Theory of Cryptographic Hash
  Functions}, Cryptology ePrint Archive, Report no.~020, 2001. Available at
  \url{http://eprint.iacr.org/2001/020}.

\bibitem{Stong}
{Stong, R.}, \emph{Solution to problem 11353}, {American Mathematical Monthly}
  {117} (2010), pp.~91--92.

\bibitem{ST}
{Suzuki,K. and Tonien, D. and Kurosawa, K. and Toyota, K.}, \emph{Birthday
  Paradox for Multicollisions}. Information Security and Cryptology - ICISC
  2006, Lecture Notes in Computer Science 4296, pp.~29--40. Springer, Berlin,
  2006.

\bibitem{vonMises}
{von Mises, R.}, \emph{\"Uber Aufteilungs- und Besetzungswahrscheinlichkeiten},
  Revue de la Facult$\acute e$ des Sciences de l' Universit$\acute e$
  d'Istanbul {4} (1939), pp.~145--163. reprinted in \cite{vM}.

\bibitem{vM}
\bysame, \emph{Selected Papers of Richard von Mises, vol. 2}, pp.~313--334.
  American Mathematical Society, Providence, R.I., 1964.

\bibitem{Wi}
{Wiener, M.J.}, \emph{Bounds on Birthday Attack Times}, Cryptology ePrint
  Archive, Report no.~318, 2005. Available at
  \url{http://eprint.iacr.org/2005/318}.

\end{thebibliography}
\end{document}